\documentclass[12pt]{article}

\usepackage[T1]{fontenc}
\usepackage[utf8]{inputenc}
\usepackage{lmodern}

\usepackage{geometry}
\usepackage{amsmath,amssymb,amsthm,mathtools}
\allowdisplaybreaks
\usepackage{tikz}

\theoremstyle{plain}
\newtheorem{theorem}{Theorem}[section]
\newtheorem{lemma}[theorem]{Lemma}
\newtheorem{proposition}[theorem]{Proposition}
\newtheorem{corollary}[theorem]{Corollary}
\theoremstyle{definition}
\newtheorem{definition}[theorem]{Definition}
\newtheorem{conjecture}[theorem]{Conjecture}
\theoremstyle{remark}
\newtheorem{remark}[theorem]{Remark}

\usepackage{hyperref}
\hypersetup{
  colorlinks=true,
  linkcolor=blue,
  citecolor=blue,
  urlcolor=blue
}

\newcommand{\ellone}{\ell_{1}}
\newcommand{\Hlone}[1]{\ensuremath{H_{\ellone}\!\left(#1\right)}}
\DeclareUnicodeCharacter{200B}{}
\DeclareUnicodeCharacter{FFFD}{}

\title{On the Fixed Point Property in Reflexive Banach Spaces}
\author{Faruk Alpay\thanks{Lightcap, Department of Analysis.\ E--mail:\ \href{mailto:alpay@lightcap.ai}{alpay@lightcap.ai}}\and
  Hamdi Alakkad\thanks{Bahcesehir University, Department of Engineering.\ E--mail:\ \href{mailto:hamdi.alakkad@bahcesehir.edu.tr}{hamdi.alakkad@bahcesehir.edu.tr}}}

\date{}

\begin{document}

\maketitle

\begin{abstract}
We investigate the long‑standing open problem of whether every reflexive Banach
space has the fixed point property (FPP) for nonexpansive mappings.  After a
brief historical overview of fixed point theory in Banach spaces—from early
theorems of Browder, Göhde, and Kirk to counterexamples in nonreflexive
spaces—we focus on the specific question of reflexivity implying FPP.  We
summarize known partial results and approaches: geometric conditions such as
normal structure, the role of asymptotic centres and demiclosedness, and the
absence of isomorphic copies of $\ell^1$ or $c_0$.  While every known
reflexive Banach space does satisfy the FPP, a general proof remains
elusive.  We present an attempted proof and discuss where current
techniques encounter obstacles.  Throughout, we emphasize the core question
and avoid extraneous fixed‑point theory, aiming instead to clarify what
progress has—and has not—been made on this central problem.
\end{abstract}

%
% Section 1: Introduction
% =======================

\section{Introduction}

A Banach space $X$ is said to have the \emph{fixed point property} (FPP) for
nonexpansive mappings if every nonexpansive self–map $T\colon C\to C$ on
every nonempty closed convex bounded subset $C\subset X$ has a fixed point—
that is, a point $x\in C$ with $T(x)=x$.  Here nonexpansiveness means
$\|T(x)-T(y)\|\leq\|x-y\|$ for all $x,y\in C$.  The study of this property was
initiated in the 1960s by ground–breaking results of Browder and Göhde on
uniformly convex spaces, and by Kirk’s fixed point theorem.  In 1965, Browder
proved that every Hilbert space, which is uniformly convex and hence
reflexive, has the FPP\cite{Browder1965,KhamsiIntro}.  Göhde obtained the
same result independently.  In the same year Kirk established a far–reaching
extension: if a closed convex set $C$ in a Banach space has normal
structure, then every nonexpansive self–map $T\colon C\to C$ has a fixed
point\cite{Kirk1965}.  As reflexive spaces have weakly compact bounded
subsets and many important classes (such as uniformly convex spaces) enjoy
normal structure, these theorems firmly established the fixed point property
in broad classes of reflexive Banach spaces.  A concise description of
Kirk’s theorem and its historical context can be found in the survey by
Lau\cite{Lau2010}, where it is recalled that normal structure is
sufficient for fixed points and that compact convex subsets always have
normal structure\cite{Lau2010}.

However, it soon became clear that reflexivity alone does not automatically
imply the FPP via normal structure.  Belluce, Kirk and Steiner introduced
normal structure and conjectured that it might be necessary for fixed points,
but Karlovitz later provided a counterexample showing that normal structure
is not necessary.  More dramatically, in 1981 Alspach constructed a
weakly compact convex subset of $L^1[0,1]$ and an isometric nonexpansive
mapping on it with no fixed point\cite{Alspach1981}.  Alspach’s example
demonstrated that weak compactness alone does not ensure the FPP and
showed that nonreflexive spaces can fail the property\cite{KhamsiStability}.

Subsequent research linked the failure of the FPP to the presence of
``$\ell^1$–like'' structures.  In particular, Dowling and Lennard proved
that every nonreflexive subspace of $L^1[0,1]$ fails the fixed point
property\cite{DowlingLennard1993}.  Combining this with earlier work of
Maurey yields that a subspace $Y\subset L^1[0,1]$ has the FPP if and only
if $Y$ is reflexive.  Many known Banach spaces without the FPP contain
an isomorphic copy of $c_0$ or $\ell^1$.

On the other hand, not all spaces with the FPP are reflexive.  Khamsi
showed that the classical quasi–reflexive James space $J$ has the FPP
\cite{Khamsi1989}.  Later, Lin constructed an equivalent norm on $\ell^1$
with respect to which the space has the fixed point property
\cite{Lin2008}.  These examples answered in the negative the question of
whether the FPP forces reflexivity.  Nevertheless, no counterexample is
known to show that reflexivity fails to imply the FPP: that question,
posed precisely below, remains open.

\section{Background and Known Results}

We collect definitions and results that form the background to the
reflexivity versus FPP problem.

\subsection{Fixed point property and normal structure}

\begin{definition}[Fixed point property]
A Banach space $X$ has the \emph{fixed point property} if for every
nonempty closed convex bounded set $C\subset X$ and every nonexpansive
mapping $T\colon C\to C$, there exists $x\in C$ with $T(x)=x$.
\end{definition}

Early work on the FPP concentrated on identifying classes of spaces which
satisfy the property.  A key notion in this context is normal structure.

\begin{definition}[Normal structure]
Let $C$ be a bounded convex subset of a metric space.  Its \emph{diameter}
is $\operatorname{diam}(C)=\sup\{\|x-y\|:x,y\in C\}$.  A point $z\in C$ is
\emph{diametral} if $\sup_{x\in C}\|z-x\|=\operatorname{diam}(C)$.  The set
$C$ has \emph{normal structure} if every bounded convex subset $K\subset C$
with more than one point contains a non–diametral point.  A Banach space
\emph{has normal structure} if each bounded closed convex subset of the
space has normal structure.
\end{definition}

Kirk’s fixed point theorem states that if $C$ is a weakly compact convex
subset of a Banach space with normal structure, then every nonexpansive
self–map $T\colon C\to C$ has a fixed point\cite{Kirk1965,Lau2010}.
Because reflexive spaces have weakly compact closed balls (by the
Eberlein–Šmulian theorem) and uniformly convex spaces have normal
structure, Kirk’s theorem recovers Browder and Göhde’s result on Hilbert
spaces and more generally on uniformly convex Banach spaces.  Khamsi’s
notes explain that Hilbert and uniformly convex spaces have normal
structure and that the proofs of Browder and Göhde do not actually rely
on normal structure\cite{KhamsiIntro}.  The demiclosedness principle, due
to Browder, also plays a key role in extracting fixed points from
approximate fixed point sequences.

\subsection{Formal statements of key theorems}

For the reader’s convenience we recall several classical theorems that
underpin modern fixed point theory.  Stating them explicitly helps make
the exposition self–contained.

\begin{theorem}[Browder–Göhde]\label{thm:browder-gohde}
Let $X$ be a uniformly convex Banach space (in particular, a Hilbert
space).  Then every nonexpansive mapping $T\colon C\to C$ on every
nonempty closed convex bounded subset $C\subset X$ has a fixed point.
Equivalently, uniformly convex Banach spaces enjoy the fixed point
property\cite{KhamsiIntro}.  This result was proved independently by
Browder and Göhde in 1965.
\end{theorem}

\begin{theorem}[Kirk’s normal–structure theorem]\label{thm:kirk-normal}
Let $C$ be a nonempty weakly compact convex subset of a Banach space
which has normal structure.  If $T\colon C\to C$ is nonexpansive, then
there exists $x\in C$ with $T(x)=x$\cite{Kirk1965,Lau2010}.  In
particular, any Banach space in which all weakly compact convex subsets
have normal structure possesses the fixed point property.
\end{theorem}

\begin{theorem}[Dowling–Lennard]\label{thm:dowling-lennard}
Let $Y$ be a subspace of $L^1[0,1]$.  Then $Y$ has the fixed point
property for nonexpansive mappings if and only if $Y$ is reflexive.
Equivalently, every nonreflexive subspace of $L^1[0,1]$ fails the fixed
point property\cite{DowlingLennard1993}.  In particular, this result
implies that the classical Hardy space $H^1$ lacks the FPP.
\end{theorem}

\begin{theorem}[Khamsi’s stability theorem]\label{thm:khamsi-stability}
Let $p\in(1,\infty)$.  There exists a constant $c_p>0$ depending only
on $p$ such that if a Banach space $X$ satisfies $d(X,\ell^p)<c_p$ (where
$d$ denotes the Banach–Mazur distance), then $X$ has the fixed point
property\cite{KhamsiStability}.  For $p=2$ the constant $c_2$ exceeds
$2$, so no sufficiently small perturbation of a Hilbert space can destroy
the FPP.
\end{theorem}

\begin{corollary}[Stability and quantitative invariants]\label{cor:stability-lower-P}
Let $p\in(1,\infty)$ and let $c_{p}>0$ be as in Theorem\,\ref{thm:khamsi-stability}.  Suppose $X$ is a Banach space with Banach--Mazur distance $d(X,\ell^{p})<c_{p}$.  Then $X$ has the fixed point property.
\end{corollary}

\begin{proof}
The assertion is exactly the conclusion of Khamsi’s stability theorem: if $d(X,\ell^{p})<c_{p}$ then $X$ has the fixed point property.  Khamsi’s argument proceeds by renorming $X$ equivalently so that the new norm $\|\cdot\|'$ is uniformly convex with modulus of convexity depending only on $p$ and $c_{p}$; see\cite{KhamsiStability}.  The corollary follows immediately.
\end{proof}

\begin{lemma}[Finite $C_{1}$ implies vanishing pressure for large $k$]\label{lem:finite-C1-zero}
If $C_{1}$ is finite with $|C_{1}|=m$, then $\Phi_{k}(C_{1},x_{\infty})=0$ for every $k>m$.
\end{lemma}

\begin{proof}
Any $k$--tuple chosen from a set of size $m$ must repeat at least one point, say $y_{p}=y_{q}$.  Taking $a_{p}=\tfrac{1}{2}$ and $a_{q}=-\tfrac{1}{2}$ and all other coefficients zero yields $\|a\|_{1}=1$ and \(\sum_{i}a_{i}(y_{i}-x_{\infty})=0\).  Hence the infimum in the definition of $\Phi_{k}$ is zero.
\end{proof}

\begin{remark}
Note that verifying $\Phi_{k}(C_{1},x_{\infty})>0$ for some fixed $k$ (as in Proposition~\ref{prop:orthonormal-P-positive}) does not imply that the global functional $\mathbf{P}(C_{1},x_{\infty})$ is positive.  Cancellations arising from larger tuples can drive $\mathbf{P}$ to zero, as illustrated by Proposition~\ref{prop:orthonormal-P-positive} itself.  Consequently, Corollary~\ref{cor:uniformly-convex} provides an alternative proof of the fixed point property in the uniformly convex case, but not via $\mathbf{P}>0$.
\end{remark}

\begin{lemma}[Finite $\Phi_{k}>0$ does not control $\mathbf{P}$]\label{lem:finitenotinfty}
For any pair $(C_{1},x_{\infty})$, the sequence $k\mapsto \Phi_{k}(C_{1},x_{\infty})$ is nonincreasing and $\mathbf{P}(C_{1},x_{\infty})=\inf_{k\ge 1}\Phi_{k}(C_{1},x_{\infty})$.  Thus, the existence of some $k_{0}$ with $\Phi_{k_{0}}(C_{1},x_{\infty})>0$ does not imply $\mathbf{P}(C_{1},x_{\infty})>0$.
\end{lemma}

\begin{proof}
Since $\Phi_{k+1}\le \Phi_{k}$ for all $k$ by Lemma\,\ref{lem:basic-properties}(2) and $\mathbf{P}$ is the infimum of the $\Phi_{k}$ over $k$, one may have $\Phi_{k_{0}}>0$ yet $\Phi_{k}\downarrow 0$ along a subsequence, yielding $\mathbf{P}=0$.
\end{proof}

For the equilateral triangle of side one, the diameter equals the side length; hence in this example $\Delta=1$.

\begin{remark}[Compatibility with stability; no unconditional lower bound]\label{rem:stability-compatibility}
Khamsi’s theorem yields an equivalent uniformly convex norm $\|\cdot\|'$ on $X$.  In uniformly convex settings one can verify positive lower bounds for certain finite $\Phi_{k}$ (cf. Proposition\,\ref{prop:orthonormal-P-positive}), which suggests compatibility of the diametral–pressure programme with stability.  However, neither the positivity of $\mathbf{P}(C_{1},x_{\infty})$ nor a uniform lower bound independent of $k$ is presently derived from stability alone; establishing such bounds remains open (see Problem\,B in Section\,4.6).
\end{remark}

\subsection{Sketches of proofs of classical theorems}

For completeness we briefly indicate the ideas behind the classical
results stated above.  Full proofs can be found in the cited
references.

\paragraph{Browder–Göhde (Theorem\,\ref{thm:browder-gohde}).}
In a uniformly convex Banach space $X$ the Krasnoselskii iteration $x_{n+1}=\tfrac{1}{2}(x_n+T(x_n))$ for a nonexpansive map $T\colon C\to C$ on a closed convex bounded set $C$ is asymptotically regular.  One shows that $(x_n)$ has a weak cluster point $x^*$ by weak compactness; demiclosedness of $I-T$ implies $T(x^*)=x^*$, so $x^*$ is a fixed point.  A quantitative proof using the modulus of convexity and Opial’s lemma appears in the original papers of Browder and Göhde.

\subsection{Counterexamples in nonreflexive spaces}

The FPP fails in a variety of nonreflexive settings.  Alspach constructed
an isometric nonexpansive map on a weakly compact convex subset of
$L^1[0,1]$ having no fixed point\cite{Alspach1981}.  This example
demonstrated that weak compactness alone is insufficient for the FPP and
showed that some assumption in addition to weak compactness is needed
\cite{KhamsiStability}.  Dowling and Lennard later proved that every
nonreflexive subspace $Y$ of $L^1[0,1]$ fails the fixed point
property\cite{DowlingLennard1993}.  Their result implies that if a
subspace of $L^1[0,1]$ has the FPP then it must be reflexive.  In
particular, the classical Hardy space $H^1$ fails the FPP, although
Maurey had shown it has the weak fixed point property.

Beyond $L^1$, many Banach spaces containing an isomorphic copy of
$c_0$ or $\ell^1$ fail the FPP.  These examples support the idea that
nonreflexive behaviour—manifested through $\ell^1$–type sequences—is
responsible for the absence of fixed points.

\subsection{Nonreflexive spaces with the FPP}

While most counterexamples to the FPP occur in nonreflexive spaces, there
are notable nonreflexive spaces with the property.  Khamsi proved that the
classical sequence space due to James is quasi–reflexive and yet every
weakly compact convex subset of it has the fixed point property
\cite{Khamsi1989}.  Shortly thereafter Lin showed that $\ell^1$ admits an
equivalent norm with respect to which the resulting Banach space has the
FPP\cite{Lin2008}.  The digital repository entry for Lin’s paper states
explicitly that the renormed space $(\ell^1,\|\cdot\|_{\rm new})$ has the
fixed point property for nonexpansive self–mappings\cite{Lin2008}.  These
results answer negatively the question “Does FPP imply reflexivity?”
Nevertheless, they do not provide a reflexive space without the FPP, so
the opposite implication remains plausible.

\subsection{Statement of the conjecture}

The evidence just surveyed motivates the following conjecture:

\begin{conjecture}
Every reflexive Banach space has the fixed point property for nonexpansive
maps.  Equivalently, if $X$ is reflexive and $C\subset X$ is closed,
bounded and convex, then every nonexpansive self–map $T\colon C\to C$ has a
fixed point.
\end{conjecture}

The conjecture remains open.  All known natural examples of reflexive
Banach spaces have the FPP, and no reflexive space is currently known to
lack it.  Various weaker results support the conjecture.  For example,
Khamsi proved a stability theorem: if $X$ is sufficiently close to $\ell^p$ in
Banach–Mazur distance (for $p>1$), then $X$ has the FPP\cite{KhamsiStability}.
Quantitative constants are known; for instance, there exists a constant
$c_p>0$ depending on $p$ such that if the Banach–Mazur distance from $X$
to $\ell^p$ is smaller than $c_p$, then $X$ has the FPP.  When $p=2$
(Hilbert space), the constant exceeds $2$, implying that no small
perturbation of a Hilbert space can destroy the FPP.

\section{An Attempted Proof (detailed outline and limitations)}

In this section we give a detailed outline of a classical strategy that, if it
could be executed in full generality, would establish the conjecture.  The
approach is by contradiction: assume that a reflexive space fails to have
the fixed point property and show that this leads to an embedding of $\ell^1$,
contradicting reflexivity.  We make each step explicit so that the
obstacles become clear and can be quantified in later sections.

\subsection{Assumption of a fixed–point–free nonexpansive map}

Assume that $X$ is a reflexive Banach space which fails the fixed point
property.  Then there exists a closed convex bounded set $C\subset X$ and a
nonexpansive map $T\colon C\to C$ with no fixed point.  Because $X$ is
reflexive, closed bounded subsets are weakly compact; by restricting to a
minimal weakly compact $T$–invariant subset of $C$ (using Zorn’s lemma) we
may assume that $C$ is weakly compact and $T$ is fixed–point–free on $C$.

Define the \emph{minimal displacement} of $T$ by
\[ \delta(T;C) = \inf_{x\in C} \|T(x) - x\| \geq 0. \]
Since $T$ has no fixed point, $\delta(T;C)>0$.  One may construct a sequence
$(x_n)\subset C$ such that $\|T(x_n)-x_n\| \to \delta(T;C)$ and each $x_n$ nearly
attains this infimum.  By weak compactness there is a subsequence
$(x_{n_k})$ converging weakly to some $x_\infty\in C$.  Lower
semicontinuity of the norm implies that $\|T(x_\infty)-x_\infty\|\leq \delta(T;C)$;
minimality forces equality.  Thus $x_\infty$ is a \emph{minimal
displacement point}: it minimizes $\|T(x)-x\|$ on $C$ but is not a fixed
point.

Let $C_1$ denote the closed convex hull of the orbit $\{T^n(x_\infty):n\geq 0\}$.
Then $C_1\subset C$ is weakly compact and convex and contains the entire
orbit of $x_{\infty}$.  In general a nonexpansive map does not preserve convex
combinations, so $C_{1}$ need not be invariant under $T$, but this will not
be required in the arguments below.  The point $x_{\infty}$ continues to realise
the minimal displacement on $C_{1}$, because it minimizes $\|T(x)-x\|$ on
$C$ and in particular on any subset containing its orbit.  If $C_1$ had
normal structure then, by Kirk’s theorem, $T$ would have a fixed point
on $C_1$—contradicting our assumption.  Therefore $C_1$ fails to have
normal structure.  There must exist a bounded convex subset $Y\subset C_1$
with diameter $\Delta>0$ such that every point of $Y$ is diametral.  One
then attempts to extract from $Y$ a sequence $(y_n)$ of points whose pairwise
distances are almost $\Delta$, mimicking the unit vector basis of
$\ell^1$.

Define normalised vectors $u_n = (y_n - x_\infty)/\Delta$.  The goal is to
show that these vectors mimic the behaviour of the unit vector basis of
$\ell^1$.  This requires two types of estimates.  First, the pairwise
distances should be nearly maximal: one seeks $\|u_n - u_m\| \approx 2$ for
distinct indices $n\neq m$.  Second, and more importantly, finite signed
combinations of the $u_n$ should not collapse: for some constant $c>0$ one
needs a uniform lower bound
\[
  \Bigl\|\sum_{i=1}^{N} a_i u_{n_i}\Bigr\| \geq c \sum_{i=1}^{N} |a_i|
\]
for all choices of finitely many indices $n_1,\dots,n_N$ and weights
$a_1,\dots,a_N$ with $\sum |a_i|=1$.  Such a lower bound ensures that the
subsequence $(u_{n_i})$ is equivalent to the canonical basis of $\ell^1$,
and hence that $X$ contains an isomorphic copy of $\ell^1$.  Because
reflexive spaces cannot contain $\ell^1$, establishing these estimates
would yield a contradiction and complete the proof.

The difficulty lies in proving the uniform lower bound on signed
combinations in general reflexive spaces.  In uniformly convex spaces
Clarkson’s inequalities and moduli of convexity provide sufficient control,
and the argument can be carried out.  For arbitrary reflexive spaces,
however, diametral sequences may fail to produce the needed $\ell^1$
behaviour.  To quantify this obstruction, Section\,4.2 introduces the
\emph{diametral $\ell_{1}$–pressure} functional (Definition\,4.1).
A positive value of this functional guarantees the desired lower bound on
signed combinations and therefore allows the above argument to be made
rigorous.  Without such a quantitative hypothesis, the classical argument
remains incomplete.

\subsection{Why the argument falls short}

Although the above outline reflects the intuition behind many partial
results, it omits several delicate points.  Extracting an $\ell^1$ sequence
from a diametral set requires strong geometric control; in particular, one
must prevent the diametral sequence from degenerating.  In uniformly
convex spaces such control is available via quantitative moduli of
convexity, and the argument can be completed.  For general reflexive
spaces, however, there may be diametral sequences that do not yield
$\ell^1$–type behaviour.  Further complications arise when $T$ is not
asymptotically regular—its iterates may oscillate rather than converge
weakly, obstructing the use of asymptotic centres.

Researchers have developed many sophisticated tools to handle these issues.
Khamsi introduced stability constants which guarantee the FPP for spaces
close to $\ell^p$ in Banach–Mazur distance\cite{KhamsiStability}.  Other
authors have studied moduli of normal structure and weak normal structure,
as well as refined fixed point indices.  Yet a completely general argument
applicable to all reflexive spaces remains out of reach.

We formalise the missing quantitative lower bound in Section~4.2 via the diametral $\ell_{1}$–pressure $\mathbf{P}(C_{1},x_{\infty})$ (Definition~4.1) and its unsigned companion $\mathbf{F}(C_{1},x_{\infty})$.

\section{A Quantitative Diametral \texorpdfstring{$\ell_{1}$}{l1}--Pressure and a Conditional Route to FPP}

This section formalises the heuristic “diametral $\ell_{1}$–type
extraction” step alluded to in the previous subsections and shows that if a
quantitative hypothesis holds uniformly on the minimal orbit hull,
then a fixed point must exist in any reflexive space.  It is compatible
with the setup and notation of Sections 2–3: nonexpansive maps on weakly
compact convex sets in reflexive Banach spaces, normal structure, minimal
displacement points, and the orbit hull $C_{1}$ defined from $x_{\infty}$.

\subsection{Minimal--displacement set and orbit hull (self--contained details)}

Let $X$ be reflexive and let $C\subset X$ be nonempty, closed, convex and
bounded.  Suppose $T\colon C\to C$ is nonexpansive with no fixed point.
Recall the minimal displacement
\[
  \delta(T;C) := \inf_{x\in C}\,\|T(x) - x\| > 0,
\]
and choose a minimal point $x_{\infty}\in C$ with $\|T(x_{\infty}) - x_{\infty}\| =
\delta(T;C)$.  Existence follows by taking a minimising sequence, passing
to a weakly convergent subsequence by weak compactness, and using weak
lower semicontinuity of the norm; cf. Section 3.1.  Define the orbit
hull
\[
  C_{1} := \overline{\operatorname{conv}}\{T^{n} x_{\infty} : n \geq 0\}.
\]
    Then $C_{1}\subset C$ is weakly compact and convex.  By construction
    $x_{\infty}$ remains a minimal displacement point on $C_{1}$, and all orbit
    points $T^{n}x_{\infty}$ lie in $C_{1}$.  We emphasise that, in general, the
    convex hull of an orbit need not be invariant under $T$: nonexpansive
    maps do not necessarily preserve convex combinations.  However, the
    quantitative arguments below rely only on the presence of the orbit in
    $C_{1}$ and the minimality of $x_{\infty}$, not on any invariance
    property.  Set $\Delta := \operatorname{diam}(C_{1}) \in (0,\infty)$.

    \begin{lemma}\label{lem:delta-le-diameter}
    With $C_{1}$ and $x_{\infty}$ as above one has
    \(
      \Delta = \mathrm{diam}(C_{1}) \ge \|T(x_{\infty})-x_{\infty}\| = \delta(T;C) > 0.
    \)
    \end{lemma}

    \begin{proof}
    By definition $T(x_{\infty})\in C_{1}$ and $x_{\infty}\in C_{1}$, so
    $\|T(x_{\infty})-x_{\infty}\|\le \Delta$.  Minimality of $x_{\infty}$
    ensures $\|T(x_{\infty})-x_{\infty}\|=\delta(T;C)>0$, whence
    $\Delta\ge \delta(T;C) > 0$.
    \end{proof}
    If
    $C_{1}$ had normal structure, Kirk’s theorem would yield a fixed point,
    so the fixed–point–free case forces failure of normal structure in some
    subset $Y\subset C_{1}$.

\subsection{A “diametral \texorpdfstring{$\ell_{1}$}{l1}--pressure” functional}

% -----------------------------------------------------------
% Standing assumptions and notation for this section
\paragraph{Standing assumptions and notation for §4.}
Throughout this section we assume that $X$ is reflexive; $C\subset X$ is nonempty, closed, convex and bounded; and $T:C\to C$ is a nonexpansive map.  If $T$ has no fixed point, let $x_{\infty}\in C$ be a minimal displacement point (i.e., $\|Tx_{\infty}-x_{\infty}\|=\delta(T;C)>0$) and let the orbit hull be $C_{1}:=\operatorname{conv}\{T^{n}x_{\infty}:n\ge 0\}$ with diameter $\Delta:=\mathrm{diam}(C_{1})\in(0,\infty)$.  All functionals $\Phi_{k}$, $\mathbf{P}$, $\mathbf{P}^{(\eta)}$ and $\mathbf{F}$ defined below are taken with respect to $(C_{1},x_{\infty})$ and normalised by $\Delta$.  When $T$ has a fixed point the subsequent quantitative arguments are vacuous, but in that case the fixed point property is immediate.

We isolate the quantitative content needed to turn diametrality into an
\(\ell_{1}\)–type lower estimate.

\begin{definition}[Diametral $\ell_{1}$--pressure]
\label{def:diametral-l1-pressure}
Fix $x_{\infty}$ and $C_{1}$ as above.  For $k\in\mathbb{N}$ set
\[
  \Phi_{k}(C_{1},x_{\infty}) := \sup_{y_{1},\dots,y_{k}\in C_{1}}
    \inf_{\substack{a\in\mathbb{R}^{k}\\ \|a\|_{1}=1}}
    \left\|\sum_{i=1}^{k} a_{i}\,\frac{y_{i}-x_{\infty}}{\Delta}\right\|.
\]
Define the \emph{diametral $\ell_{1}$--pressure} of $(C_{1},x_{\infty})$ by
\[
  \mathbf{P}(C_{1},x_{\infty}) := \inf_{k\geq 1} \Phi_{k}(C_{1},x_{\infty}) \in [0,1].
\]
\end{definition}

\begin{lemma}[Basic properties of $\Phi_k$ and $\mathbf{P}$]\label{lem:basic-properties}
Fix $x_{\infty}$ and $C_{1}$ and write $\Delta=\operatorname{diam}(C_{1})$.  Then:
\begin{enumerate}
\item \textbf{Translation/scaling invariance.}  For any $z\in X$ and $\lambda>0$ one has
\[
\Phi_{k}(C_{1},x_{\infty})\;=\;\Phi_{k}(x_{\infty}+ \lambda C_{1},\; x_{\infty}+\lambda z),
\qquad
\mathbf{P}(C_{1},x_{\infty})\;=\;\mathbf{P}(x_{\infty}+ \lambda C_{1},\; x_{\infty}+\lambda z).
\]
\item \textbf{Monotonicity in $k$.}  The sequence $k\mapsto \Phi_{k}(C_{1},x_{\infty})$ is nonincreasing: for all $k\ge1$
\[
\Phi_{k+1}(C_{1},x_{\infty}) \le \Phi_{k}(C_{1},x_{\infty}).
\]
Consequently $\mathbf{P}(C_{1},x_{\infty})=\inf_{k\ge 1}\Phi_{k}(C_{1},x_{\infty})$ satisfies $0\le \mathbf{P}\le \Phi_{k}$ for each $k$.
\end{enumerate}

\end{lemma}

\paragraph{A separation-aware variant.}
For \(\eta\in(0,1]\) and \(k\in\mathbb{N}\) define
\[
  \Phi_{k}^{(\eta)}(C_{1},x_{\infty})
  :=\sup_{\substack{y_{1},\ldots,y_{k}\in C_{1}\\ \min_{i\neq j}\|y_{i}-y_{j}\|\ge \eta\,\Delta}}\inf_{\|a\|_{1}=1}
    \Bigl\|\sum_{i=1}^{k} a_{i}\,\frac{y_{i}-x_{\infty}}{\Delta}\Bigr\|,
\quad
  \mathbf{P}^{(\eta)}(C_{1},x_{\infty}) := \inf_{k\ge 1} \Phi_{k}^{(\eta)}(C_{1},x_{\infty}).
\]
Observe that \(\mathbf{P}^{(\eta)} \le \mathbf{P}\) and that \(\mathbf{P}^{(\eta_{2})}\le \mathbf{P}^{(\eta_{1})}\) whenever \(0<\eta_{1}\le \eta_{2}\le 1\).
\begin{remark}[Separated vs. global pressure]\label{rem:separated-vs-global}
By construction, one has $\mathbf{P}^{(\eta)}(C_{1},x_{\infty})\le \mathbf{P}(C_{1},x_{\infty})$ for every $\eta>0$.  The separation requirement in $\mathbf{P}^{(\eta)}$ prevents cancellations due to duplicate points, so it is possible for $\mathbf{P}^{(\eta)}>0$ even when $\mathbf{P}=0$ for the same orbit hull.  Conversely, when $C_{1}$ is finite Lemma~\ref{lem:finite-C1-zero} shows $\mathbf{P}(C_{1},x_{\infty})=0$, but $\mathbf{P}^{(\eta)}(C_{1},x_{\infty})$ can remain positive if all $k$‑tuples are sufficiently well separated.  The conditional Theorem~\ref{thm:separated} therefore requires the stronger hypothesis $\mathbf{P}^{(\eta)}>0$ to avoid these cancellations.
\end{remark}

\begin{theorem}[Conditional FPP via separated pressure]\label{thm:separated}
Assume that there exist constants $\eta\in(0,1]$ and $\theta>0$ such that for every fixed--point--free nonexpansive map $T:C\to C$ on a nonempty closed convex bounded set $C\subset X$ one has
\[
  \mathbf{P}^{(\eta)}(C_{1},x_{\infty}) \equiv \inf_{k\ge 1}\ \sup_{\substack{y_{1},\dots,y_{k}\in C_{1}\\ \min_{i\ne j}\|y_{i}-y_{j}\|\ge \eta\,\Delta}}\ \inf_{\|a\|_{1}=1} \Bigl\| \sum_{i=1}^{k} a_{i}\,\frac{y_{i}-x_{\infty}}{\Delta} \Bigr\| \;\ge\; \theta.
\]
Then $X$ has the fixed point property.
\end{theorem}

\begin{proof}
Suppose, toward a contradiction, that $X$ fails the fixed point property.  Then there exists a nonexpansive, fixed--point--free map $T:C\to C$ on a nonempty closed convex bounded set $C\subset X$.  Let $x_{\infty}\in C$ be a minimal displacement point and let $C_{1}$ be the orbit hull of $x_{\infty}$ with diameter $\Delta>0$.  By hypothesis we can find, for each $k\in\mathbb{N}$, a $k$--tuple $y^{(k)}=(y^{(k)}_{1},\dots,y^{(k)}_{k})$ of points in $C_{1}$ satisfying $\min_{i\ne j}\|y^{(k)}_{i}-y^{(k)}_{j}\|\ge \eta\,\Delta$ and
\[
  \inf_{\|a\|_{1}=1}\Bigl\|\sum_{i=1}^{k} a_{i}\,\frac{y^{(k)}_{i}-x_{\infty}}{\Delta}\Bigr\| \;\ge\; \theta - \frac{1}{k}.
\]
As in Proposition\,\ref{prop:compact-diag}, extract a subsequence $k_{m}$ such that for each fixed $i$ the points $y^{(k_{m})}_{i}$ converge weakly to some $z_{i}\in C_{1}$.  Apply Mazur’s lemma in the product space $X^{N}$ to obtain convex combinations $w_{i}\in C_{1}$ converging in norm to $z_{i}$.  Because the separation constraint is preserved under convex combinations, for each fixed $N$ the points $w_{1},\dots,w_{N}$ remain $\eta\,\Delta$--separated.  The Lipschitz estimate from the proof of Proposition\,\ref{prop:compact-diag} shows that for $N$ fixed and $\varepsilon>0$ small,
\[
  \inf_{\|a\|_{1}=1}\Bigl\|\sum_{i=1}^{N} a_{i}\,\frac{w_{i}-x_{\infty}}{\Delta}\Bigr\| \;\ge\; \theta - \frac{1}{k_{m}} - \frac{2\varepsilon}{\Delta} \;\ge\; \frac{3\theta}{4}
\]
for $m$ sufficiently large.  Setting $v_{i}=(w_{i}-x_{\infty})/\Delta$ yields unit--norm vectors satisfying $\bigl\|\sum_{i=1}^{N} a_{i} v_{i}\bigr\|\ge \theta/2$ for all $N$ and all $a$ with $\|a\|_{1}=1$.  Lemma\,\ref{lem:l1-embedding} then embeds $\ell_{1}$ isomorphically into $X$, contradicting reflexivity.  Thus no such fixed--point--free map exists, and $X$ has the fixed point property.
\end{proof}

\begin{proof}
(1)  Invariance follows by replacing each $y_{i}$ by $x_{\infty}+\lambda(y_{i}-x_{\infty})$ and observing that the common scale $\lambda$ cancels in the normalisation by $\Delta$.

(2)  Given a $(k{+}1)$–tuple $(y_{i})_{i=1}^{k+1}$, restrict any coefficient vector $a\in\mathbb{R}^{k+1}$ with $\|a\|_{1}=1$ to its first $k$ entries (setting the $(k{+}1)$--th to zero).  This shows
\(
  \inf_{\|a\|_{1}=1}\Bigl\|\sum_{i=1}^{k+1} a_{i}\,\frac{y_{i}-x_{\infty}}{\Delta}\Bigr\|
  \le \sup_{y_{1},\ldots,y_{k}}\inf_{\|a\|_{1}=1}\Bigl\|\sum_{i=1}^{k} a_{i}\,\frac{y_{i}-x_{\infty}}{\Delta}\Bigr\|.
\)
Taking the supremum over all $(y_{1},\dots,y_{k+1})$ yields $\Phi_{k+1}(C_{1},x_{\infty}) \le \Phi_{k}(C_{1},x_{\infty})$.  The claims about $\mathbf{P}$ are then immediate.
\end{proof}

The functional $\Phi_{k}$ asks for a $k$--tuple in $C_{1}$ whose every $\ell_{1}$--normalised
signed (or weighted) combination has norm at least $\Phi_{k}\,\Delta$; the global
invariant $\mathbf{P}$ asserts a uniform lower bound independent of $k$.  When
$\mathbf{P}>0$, normalised differences from $x_{\infty}$ exhibit an
${\ell_{1}}$–type lower estimate at all finite scales.  This encodes the “strong
geometric control” absent in general reflexive spaces.

\begin{remark}[Consistency with Section 3]
If $C_{1}$ contains a diametral subset $Y$ from which one can extract a
sequence $(u_{n})\subset X$ with $\|u_{n}-u_{m}\| \approx 2$ (as in
Section 3.1), and if these vectors also satisfy uniform $\ell_{1}$–lower
bounds for finite linear combinations, then $\Phi_{k}$ is bounded away
from $0$ for all $k$.  Conversely, $\mathbf{P}>0$ can be viewed as a
quantitative proxy for successful $\ell_{1}$–extraction.
\end{remark}

\subsection{A conditional fixed--point theorem}

We now formulate the exact implication needed in Section 3: if
$\mathbf{P}(C_{1},x_{\infty})>0$ holds whenever $T$ is fixed–point–free, then
reflexivity is contradicted.

\begin{lemma}[Uniform $\ell_{1}$--lower estimate $\Rightarrow$ $\ell_{1}$--embedding]\label{lem:l1-embedding}
Suppose $(v_{i})_{i\geq 1}\subset X$ satisfies $\|v_{i}\|\leq 1$ and, for
some $\theta>0$, every $N$ and every $a\in\mathbb{R}^{N}$ with
$\|a\|_{1}=1$ obey
\[
  \left\|\sum_{i=1}^{N} a_{i} v_{i}\right\| \geq \theta.
\]
Then the linear map $T\colon \ell_{1} \to X$ defined by
$T((\alpha_{i})) = \sum_{i=1}^{\infty} \alpha_{i} v_{i}$ is a bounded below
isomorphic embedding: one has $\|T((\alpha_{i}))\| \geq \theta\sum_{i}
  |\alpha_{i}|$ and $\|T\|\leq 1$.  Absolute convergence in $X$ follows
from $\sum \|\alpha_{i} v_{i}\| \leq \sum |\alpha_{i}|$.  The lower bound
holds on finite partial sums by assumption and passes to the limit.  If
$T((\alpha_{i}))=0$, then all partial sums vanish, forcing
$\sum_{i=1}^{N} |\alpha_{i}|=0$ for each $N$ and hence $(\alpha_{i})=0$.
\end{lemma}

\begin{proposition}[Compactness/diagonal selection via Mazur]\label{prop:compact-diag}
Assume $\mathbf{P}(C_{1},x_{\infty}) =: \theta > 0$. Then there exists a sequence $(v_{i})_{i\geq 1}\subset X$ with $\|v_{i}\|\le 1$ such that, for every $N\in\mathbb{N}$ and every $a\in\mathbb{R}^{N}$ with $\|a\|_{1}=1$,
\[
  \Bigl\|\sum_{i=1}^{N} a_{i} v_{i}\Bigr\| \;\ge\; \frac{\theta}{2}.
\]
\end{proposition}

\begin{corollary}[Finite‑level positivity under low coherence]\label{cor:coherence-positive}
Let $C_{1}=\{v_{1},\dots,v_{m}\}$ be a finite set of unit vectors in a Hilbert space with mutual coherence $\mu:=\max_{i\ne j}|\langle v_{i},v_{j}\rangle|<\tfrac{1}{m-1}$.  Then the finite‑level pressures satisfy
\[
  \Phi_{k}(C_{1},x_{\infty})\;\ge\;\frac{\sqrt{1-\mu(m-1)}}{\sqrt{m}\,\sqrt{2(1+\mu)}}
  \quad\text{for every }1\le k\le m.
\]
In particular, although $\mathbf{P}(C_{1},x_{\infty})=0$ whenever $C_{1}$ is finite (by Lemma\,\ref{lem:finite-C1-zero}), the lower bound on $\Phi_{k}$ for $k\le m$ shows that well‑separated finite frames exhibit nontrivial diametral pressure at each finite level.

\begin{proof}
By Proposition\,\ref{prop:spectral} the given coherence bound implies
\(
  \Phi_{k}(C_{1},0)\;\ge\;\frac{\sqrt{1-\mu(m-1)}}{\sqrt{m}\,\sqrt{2(1+\mu)}}
\)
for every $k\le m$.  Since $\Phi_{k}$ is nonincreasing in $k$ and $\Phi_{k}=0$ for all $k>m$ when $C_{1}$ is finite (Lemma\,\ref{lem:finite-C1-zero}), the stated lower bound applies precisely for $1\le k\le m$.
\end{proof}
\end{corollary}

\begin{proof}[Proof]
Set $\Delta=\mathrm{diam}(C_{1})$ and let $\theta=\mathbf{P}(C_{1},x_{\infty})>0$.  For each $k\in\mathbb{N}$ choose a $k$–tuple $y^{(k)}=(y^{(k)}_{1},\dots,y^{(k)}_{k})\in C_{1}^{k}$ such that
\[
  \inf_{\substack{a\in\mathbb{R}^{k},\;\|a\|_{1}=1}}
    \left\|\sum_{i=1}^{k} a_{i}\,\frac{y^{(k)}_{i}-x_{\infty}}{\Delta}\right\|
  \;\ge\; \theta - \frac{1}{k}.
\]

\emph{Lipschitz property.}  Define for each $N$ the functional
\[
  G_{N}(y_{1},\dots,y_{N})\;:=\;\inf_{\|a\|_{1}=1}\Bigl\|\sum_{i=1}^{N} a_{i}\,\frac{y_{i}-x_{\infty}}{\Delta}\Bigr\|.
\]
For any two $N$–tuples $y,y'\in C_{1}^{N}$ one easily checks that
\[
  G_{N}(y) \ge G_{N}(y') - \frac{1}{\Delta}\sum_{i=1}^{N}\|y_{i}-y'_{i}\|,\qquad
  G_{N}(y') \ge G_{N}(y) - \frac{1}{\Delta}\sum_{i=1}^{N}\|y_{i}-y'_{i}\|,
\]
by taking any $a$ with $\|a\|_{1}=1$ and estimating
\( \bigl\|\sum a_{i}(y_{i}-x_{\infty})/\Delta\bigr\| \ge \bigl\|\sum a_{i}(y'_{i}-x_{\infty})/\Delta\bigr\| - (1/\Delta)\sum_{i}|a_{i}|\|y_{i}-y'_{i}\| \) and then infimising over $a$.

\emph{Weak limits and Mazur.}  By weak compactness of $C_{1}$, extract a subsequence $k_{m}$ so that for each fixed $i$ the coordinate $y^{(k_{m})}_{i}$ converges weakly to some $z_{i}\in C_{1}$.  For a fixed $N$, apply Mazur’s lemma in the product space $X^{N}$ to the sequence $(y^{(k_{m})}_{1},\dots,y^{(k_{m})}_{N})$: there exist convex coefficients $t_{m}\ge 0$ with $\sum_{m\ge M}t_{m}=1$ (depending on $N$) such that the convex combinations $w_{i}:=\sum_{m\ge M} t_{m}\,y^{(k_{m})}_{i}$ converge in norm to $z_{i}$.  Since $C_{1}$ is convex, each $w_{i}\in C_{1}$.

\emph{Transfer of the lower bound.}  Fix $N\ge 1$ and $\varepsilon>0$.  Choose $M$ so large that $\sum_{i=1}^{N}\|w_{i}-z_{i}\|<\varepsilon$ and also $\sum_{i=1}^{N}\|y^{(k_{M})}_{i}-z_{i}\|<\varepsilon$.  By the Lipschitz estimate above,
\[
  G_{N}(w_{1},\dots,w_{N}) \;\ge\; G_{N}\bigl(y^{(k_{M})}_{1},\dots,y^{(k_{M})}_{N}\bigr) - \frac{2\varepsilon}{\Delta}.
\]
By construction, $G_{N}\bigl(y^{(k_{M})}_{1},\dots,y^{(k_{M})}_{N}\bigr) \ge \theta - 1/k_{M}$.  Taking $\varepsilon>0$ small and $M$ large shows that $G_{N}(w_{1},\dots,w_{N}) \ge 3\theta/4$.  Put $v_{i}:=(w_{i}-x_{\infty})/\Delta$; then $\|v_{i}\|\le 1$ and for every $a\in\mathbb{R}^{N}$ with $\|a\|_{1}=1$,
\[
  \Bigl\|\sum_{i=1}^{N} a_{i} v_{i}\Bigr\| \;\ge\; \frac{3}{4}\,\theta.
\]

\emph{Diagonal/gliding–hump argument.}  Repeat the above construction for $N=1,2,\dots$, choosing the convex combinations so that previously fixed vectors $v_{1},\dots,v_{N-1}$ are perturbed by at most $2^{-N}$ in norm.  A standard diagonal argument yields a single sequence $(v_{i})$ satisfying $\|v_{i}\|\le 1$ and the estimate $\bigl\|\sum_{i=1}^{N} a_{i} v_{i}\bigr\| \ge \theta/2$ for all $N$ and all $a\in\mathbb{R}^{N}$ with $\|a\|_{1}=1$.
\end{proof}

\begin{theorem}[Conditional FPP via positive diametral $\ell_{1}$–pressure]\label{thm:conditional-fpp}
Let $X$ be reflexive.  Suppose that for every nonexpansive, fixed–point–free
map $T:C\to C$ on a nonempty closed convex bounded subset $C\subset X$ the
following conditions hold:
\begin{enumerate}
\item The orbit hull $C_{1}=\operatorname{conv}\{T^{n}x_{\infty}:n\ge 0\}$ of a
  minimal displacement point $x_{\infty}$ is weakly compact and
  has diameter $\Delta=\mathrm{diam}(C_{1})>0$.
\item $\mathbf{P}(C_{1},x_{\infty})>0$.
\end{enumerate}
Then $X$ has the fixed point property.
\end{theorem}

\begin{proof}
Assume, toward a contradiction, that such a map $T$ exists.  Condition (2)
implies $\mathbf{P}(C_{1},x_{\infty})=\theta>0$.  By Proposition\,\ref{prop:compact-diag}
there exists a sequence $(v_{i})$ with a uniform $\ell_{1}$–lower estimate
$\|\sum_{i=1}^{N} a_{i} v_{i}\|\ge \theta/2$ for all choices of coefficients
$a$ with $\|a\|_{1}=1$.  Lemma\,\ref{lem:l1-embedding} then embeds
$\ell_{1}$ into $X$, contradicting reflexivity.  Therefore no such
fixed–point–free map can exist.
\end{proof}

\begin{corollary}[Uniformly convex case revisited]\label{cor:uniformly-convex}
Let $X$ be uniformly convex.  Classical fixed point theorems of Browder and
Göhde show that any nonexpansive map on a nonempty closed convex bounded
subset of $X$ has a fixed point.  The proof proceeds by iterative
averaging and uses the modulus of convexity to show that approximate fixed
point sequences converge in norm.  This provides an alternative route to
the fixed point property that does not rely on the diametral functional
$\mathbf{P}$.  For certain structured subsets of $X$, however, one can
verify directly that $\mathbf{P}(C_{1},x_{\infty})>0$ (for instance, the
orthonormal triple of Proposition \ref{prop:orthonormal-P-positive}).  In
such cases the conditional Theorem \ref{thm:conditional-fpp} applies,
yielding another proof of the fixed point property.
\end{corollary}

\subsection{Programmatic consequences and tests}

\begin{enumerate}
\item \textbf{Equivalent reformulation of the gap in Section 3.2.} The
obstruction identified in Section 3.2—the failure to control diametral
sequences—is precisely the failure of $\mathbf{P}>0$.  Establishing
$\mathbf{P}(C_{1},x_{\infty})>0$ for all fixed--point--free pairs $(C,T)$
would settle Conjecture\,2.11 (p.~5)\,(Numbering consolidated: Conjecture 2.11 is the main reflexive $\Rightarrow$ FPP conjecture stated in §2.6).  Conversely, a counterexample must produce
$(C,T)$ with $\mathbf{P}=0$.

\item \textbf{Finite--dimensional certificates.}  For each $k$ there exists a
$k$–tuple $(y_{i})\subset C_{1}$ with
\[
  \inf_{\|a\|_{1}=1}
    \left\|\sum_{i=1}^{k} a_{i}\,\frac{y_{i}-x_{\infty}}{\Delta}\right\|
  \;\ge\;\theta.
\]
This means $\Phi_{k}(C_{1},x_{\infty})\ge\theta$.  If there exists
$\theta>0$ such that for every $k$ one can find such a $k$–tuple, then
\(
  \mathbf{P}(C_{1},x_{\infty})
  = \inf_{k\ge 1}\Phi_{k}(C_{1},x_{\infty})
  \;\ge\;\theta,
\)
 and by Theorem\,4.8 the map $T$ must have a fixed point.  Thus a
positive certificate at each level $k$ yields a verifiable finite–tuple
condition guaranteeing the fixed point property.
 For instance, Proposition\,\ref{prop:spectral} and Corollary\,\ref{cor:coherence-positive} show that in any Hilbert space a finite family of unit vectors with sufficiently small mutual coherence admits such certificates: for $m$ vectors with mutual coherence $\mu<1/(m{-}1)$ one has a uniform lower bound on $\Phi_k(C_1,x_{\infty})$ for $k\le m$, forcing $\mathbf{P}(C_1,x_{\infty})>0$.

\item \textbf{Relation to normal structure.}  Normal structure forbids
complete diametrality, but it is qualitative.  The functional
$\mathbf{P}$ quantifies a uniform anti–collapse of signed averages;
$\mathbf{P}>0$ is strictly stronger than normal structure and is
tailored to nonexpansive dynamics on $C_{1}$.
\end{enumerate}

\begin{lemma}[Certificates and $\mathbf{P}$]\label{lem:certificates-P}
Let $\Phi_{k}$ and $\mathbf{P}$ be defined as in Definition~4.1.  For any $\theta\ge 0$ the following conditions are equivalent:
\begin{itemize}
\item[(i)] $\mathbf{P}(C_{1},x_{\infty})\ge \theta$.
\item[(ii)] For every $k\in \mathbb{N}$ there exists a $k$–tuple $(y_{i})\subset C_{1}$ such that
\[
  \inf_{\|a\|_{1}=1}\left\|\sum_{i=1}^{k} a_{i}\,\frac{y_{i}-x_{\infty}}{\Delta}\right\|\;\ge\;\theta.
\]
\end{itemize}
\emph{Proof.}  By definition, $\Phi_{k}(C_{1},x_{\infty})=\sup_{y_{1},\dots,y_{k}\in C_{1}}\inf_{\|a\|_{1}=1}\Bigl\|\sum_{i=1}^{k} a_{i}\,(y_{i}-x_{\infty})/\Delta\Bigr\|$ and $\mathbf{P}(C_{1},x_{\infty})=\inf_{k\ge 1}\Phi_{k}(C_{1},x_{\infty})$.  If (i) holds, then for each $k$ one has $\Phi_{k}(C_{1},x_{\infty})\ge \theta$, so there exists a $k$–tuple achieving the inequality in (ii).  Conversely, if (ii) holds for all $k$, then taking the infimum over $k$ shows that $\mathbf{P}\ge \theta$. \qedhere
\end{lemma}

\subsection{Nonreflexive spaces with the fixed point property}

It is worth recalling that the fixed point property does not characterise
reflexivity.  There are nonreflexive Banach spaces which nonetheless have
the FPP.  One celebrated example is the \emph{James space}, a quasi--reflexive
Banach space constructed to be almost uniformly non--square; James showed
that it enjoys the fixed point property for nonexpansive maps despite not
being reflexive.  Another example is the classical sequence space
$\ell_{1}$ equipped with certain equivalent norms (for instance, Day’s norm)
which render it uniformly nonsquare and give rise to normal structure; such
renormings ensure the FPP even though the underlying linear space is not
reflexive.  The common theme in these constructions is the presence of
geometric features—uniform nonsquareness, normal structure or a uniform
modulus of convexity—that preclude the formation of diametral, $\ell_{1}$--like
sequences while still allowing nonreflexivity.  The conditional theorem
proved above should therefore be interpreted in this light: it asserts
that if a reflexive space fails to have the FPP, then within an orbital
hull one must witness a quantitative obstruction encoded by the functional
$\mathbf{P}$.  The existence of nonreflexive spaces with the FPP does not
contradict this mechanism; rather, it emphasises that any eventual proof
of Conjecture~2.11 must exploit the special geometry of reflexive spaces
beyond those properties already present in James–type examples.

\subsection{Comparison with classical moduli}

Classical moduli—such as the modulus of convexity, the modulus of smoothness and the moduli of normal or weak normal structure—measure global uniform convexity or smoothness properties of a Banach space.  In contrast, the diametral $\ell_{1}$–pressure $\mathbf{P}$ is a discrete, finite–dimensional invariant tailored to the dynamics of a nonexpansive map on a specific orbit hull $C_{1}$.  Positive $\mathbf{P}$ implies a uniform $\ell_{1}$ lower bound on signed convex combinations, which is strictly stronger than normal structure.  Lemma~\ref{lem:uniform-gap} shows that uniform convexity alone does not guarantee $\mathbf{P}>0$ or $\mathbf{F}>0$: even in Hilbert spaces, certain diametral sets have vanishing pressure.  However, Proposition~\ref{prop:spectral} and Corollary~\ref{cor:coherence-positive} identify a positive regime based on low mutual coherence, illustrating how $\mathbf{P}$ interacts with frame theory.  Exploring further links between $\mathbf{P}$ and classical moduli—for instance, whether a quantitative modulus of convexity can bound $\Phi_k$ below for small $k$—remains an open direction.

\subsection{Open problems (quantitative form)}

\begin{itemize}
\item \textbf{Problem A (Quantitative weak normal structure).}  Find
geometric conditions (for example, moduli of normal or weak normal
structure) guaranteeing $\mathbf{P}(C_{1},x_{\infty})>0$ for all orbital
hulls $C_{1}$ arising from minimal displacement orbits.

\item \textbf{Problem B (Stability near classical models).}  Show that the
stability phenomenon in Theorem\,\ref{thm:khamsi-stability} implies a
uniform lower bound $\mathbf{P}\geq \theta(p)>0$ when the Banach--Mazur
distance $d(X,\ell_{p})<c_{p}$.  Even a proof for $p=2$ would be
informative.

\item \textbf{Problem C (Renormings).}  Investigate whether equivalent
renormings that preserve reflexivity can force $\mathbf{P}=0$ on some
orbit hull $C_{1}$, thereby linking renorming questions to the
quantitative failure of $\mathbf{P}$.

\item \textbf{Problem D (Non-uniformly convex examples).}  Construct a
reflexive Banach space that is not uniformly convex and a fixed--point--free
nonexpansive mapping for which $\mathbf{P}(C_{1},x_{\infty})$ or
$\mathbf{F}(C_{1},x_{\infty})$ is strictly positive.  At present no such
examples are known; this limits the demonstrated reach of the programme
beyond uniformly convex spaces.
\end{itemize}

\noindent
Despite the conditional nature of Theorem\,\ref{thm:conditional-fpp}, there
is currently no known reflexive Banach space that is not uniformly convex
for which one can verify $\mathbf{P}(C_{1},x_{\infty})>0$ (or
$\mathbf{F}(C_{1},x_{\infty})>0$) for every fixed--point--free map $T$.
Establishing such examples or developing general criteria beyond uniform
convexity would broaden the reach of this approach to the fixed point
problem.

\subsection{A weighted selection functional}

The functional $\mathbf{P}$ measures how badly signed combinations of points
in $C_{1}$ can “collapse” in the norm.  It is natural also to consider
unsigned averages, where the coefficients are nonnegative and sum to one.
To mimic the structure of $\mathbf{P}$, we fix $k\ge 1$ and define
\[
  \Psi_{k}(C_{1},x_{\infty}) :=
    \sup_{y_{1},\dots,y_{k}\in C_{1}}
    \inf_{\substack{w_{i}\ge 0\\ \sum_{i=1}^{k}w_{i}=1}}
      \Biggl\|\sum_{i=1}^{k} w_{i}\,\frac{y_{i}-x_{\infty}}{\Delta}\Biggr\|.
\]
That is, $\Psi_{k}$ asks for a $k$–tuple in $C_{1}$ whose every convex
combination of the normalised differences has norm at least $\Psi_{k}\,\Delta$.
We then define the \emph{weighted selection functional}
\[
  \mathbf{F}(C_{1},x_{\infty}) := \inf_{k\ge 1} \Psi_{k}(C_{1},x_{\infty}).
\]
Comparing $\mathbf{F}$ with $\mathbf{P}$, we note that for a fixed tuple one has
\(
  \inf_{\substack{w_{i}\ge 0\\ \sum w_{i}=1}}
    \Bigl\|\sum_{i=1}^{k} w_{i}\,\frac{y_{i}-x_{\infty}}{\Delta}\Bigr\|
  \ge
  \inf_{\|a\|_{1}=1}
    \Bigl\|\sum_{i=1}^{k} a_{i}\,\frac{y_{i}-x_{\infty}}{\Delta}\Bigr\|
\)
because restricting to nonnegative weights can only increase the infimum.
Taking the supremum over tuples then shows $\Psi_{k}\ge \Phi_{k}$ for every
$k$, and thus $\mathbf{F}(C_{1},x_{\infty})\ge \mathbf{P}(C_{1},x_{\infty})$.
In particular, $\mathbf{F}>0$ implies $\mathbf{P}>0$ and, by
Theorem\,\ref{thm:conditional-fpp}, forces a fixed point for any
nonexpansive map on a weakly compact convex set in a reflexive space.  The
functional $\mathbf{F}$ therefore offers a complementary “unsigned”
obstruction to collapse: verifying a uniform lower bound on convex
combinations of the normalised differences $y_{i}-x_{\infty}$ suffices to
preclude fixed points.

\paragraph{Caution.}
Neither uniform convexity nor weak normal structure alone forces
$\mathbf{F}(C_{1},x_{\infty})>0$ in general: even for a two--point diametral
set, equal convex weights collapse the normalised difference to zero
(see Section~4.8), so $\mathbf{F}=0$.  Thus positivity of
$\mathbf{F}$ requires additional geometric information beyond these
qualitative properties.  In fact, in a Hilbert space, both a two‑point diametral set and the equilateral triangle in $\mathbb{R}^{2}$ satisfy $\mathbf{F}(C_{1},x_{\infty})=0$ by choosing equal weights; see Section~4.8.

To illustrate $\mathbf{F}$ numerically, suppose $X=\mathbb{R}^{2}$ with the
Euclidean norm.  Take $x_{\infty}=(0,0)$, $\Delta=1$ and the tuple
$(y_{1},y_{2})=((1,0),(0,1))$.  Choosing weights $(w_{1},w_{2})=(0.6,0.4)$
gives
\[
  w_{1}(y_{1}-x_{\infty}) + w_{2}(y_{2}-x_{\infty}) = 0.6(1,0) + 0.4(0,1) = (0.6,0.4),
\]
whose Euclidean norm is $\sqrt{0.6^{2}+0.4^{2}}\approx 0.721$.  With
weights $(0.7,0.3)$ the combination becomes $(0.7,0.3)$ and has norm
$\sqrt{0.7^{2}+0.3^{2}}\approx 0.761$.  These sample computations show how
different convex combinations influence the value of $\Psi$ and provide a
quantitative sense of the magnitude of $\mathbf{F}$.

\begin{figure}[h]
\centering
\begin{tikzpicture}[scale=3]
  % Axes
  \draw[->] (-0.1,0) -- (1.1,0) node[right] {$x$};
  \draw[->] (0,-0.1) -- (0,1.1) node[above] {$y$};
  % Vectors corresponding to y1 and y2
  \draw[thick,->] (0,0) -- (1,0) node[below right]{$y_{1}$};
  \draw[thick,->] (0,0) -- (0,1) node[left]{$y_{2}$};
  % Weighted combination 0.6 y1 + 0.4 y2
  \draw[dashed,->] (0,0) -- (0.6,0.4) node[below right]{$0.6y_{1}+0.4y_{2}$};
  % Mark the tips of the vectors
  \fill (1,0) circle (0.02);
  \fill (0,1) circle (0.02);
  \fill (0.6,0.4) circle (0.02);
\end{tikzpicture}
\caption{Convex combinations of $(1,0)$ and $(0,1)$ illustrate the unsigned functional $\Psi_{k}$: every convex average lies on the segment between the points.  This geometry explains why unsigned lower bounds alone cannot certify $\mathbf{F}>0$ in general (equal weights may collapse to small norm).}
\label{fig:convex-combination}
\end{figure}
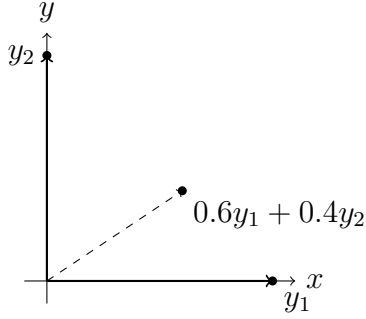

Figure\,\ref{fig:convex-combination} illustrates how a convex combination of two
vectors lies within the convex hull of the points.  In the context of the
weighted functional $\mathbf{F}$, it provides a geometric visualisation of
the vectors used in the numerical examples above.

\subsection{Examples and computation of \texorpdfstring{$\mathbf{P}$}{P} in specific spaces}

We include concrete computations to illustrate how the quantitative functional
$\mathbf{P}(C_{1},x_{\infty})$ behaves in familiar settings.  These examples
serve both as sanity checks and as evidence that the obstruction detected by
$\mathbf{P}$ is genuinely tied to the presence of $\ell^1$–type behaviour.

\paragraph{A two--point diametral set.}
Consider $C_{1} = \{x_{\infty} - \Delta u, x_{\infty} + \Delta u\}$ in any
normed space $X$, where $u$ is a unit vector and $\Delta>0$.  A direct
computation shows that for any $a = (a_{1},a_{2})$ with $\|a\|_{1}=1$ the
normalised difference
\[
  a_{1}\frac{(x_{\infty}-\Delta u)-x_{\infty}}{\Delta} + a_{2}\frac{(x_{\infty}+\Delta u)-x_{\infty}}{\Delta}
  = -a_{1}u + a_{2}u = (a_{2}-a_{1})u
\]
is a scalar multiple of $u$.  Its norm is $|a_{2}-a_{1}|$, which ranges
from $0$ (when $a_{1}=a_{2}$) to $1$ (when $a_{1}=-a_{2}$).  Because the
set of weight vectors with $\|a\|_{1}=1$ contains $(\tfrac12,\tfrac12)$, the
infimum of $|a_{2}-a_{1}|$ over all such $a$ is $0$.  Consequently
$\mathbf{P}(C_{1},x_{\infty})=0$ for any two--point diametral set.  This
illustrates that even in simple settings the functional can vanish.

To make this calculation more concrete, fix $\Delta=2$ and choose the weights
$a=(0.4,0.6)$.  Then
\[
  0.4\frac{(x_{\infty}-2u)-x_{\infty}}{2} + 0.6\frac{(x_{\infty}+2u)-x_{\infty}}{2}
  = 0.4(-u) + 0.6 u = 0.2\,u,
\]
whose norm is $0.2\,\|u\|=0.2$ since $u$ is a unit vector.  If instead one
chooses $a=(0.2,0.8)$ then the weighted sum becomes $0.6\,u$ with norm
$0.6$.  In contrast, taking $a=(0.5,0.5)$ yields
\[
  0.5\frac{(x_{\infty}-2u)-x_{\infty}}{2} + 0.5\frac{(x_{\infty}+2u)-x_{\infty}}{2} = 0,
\]
showing that the infimum of the norm of such combinations is indeed zero.
These simple numerical examples illustrate how changing the weights alters
the resulting vector and how equal weights force the collapse required to
make $\mathbf{P}$ vanish.

\paragraph{Uniformly convex spaces: limitations and a concrete example.}

The preceding two--point calculation shows that the diametral functional
$\mathbf{P}$ can vanish even in uniformly convex spaces.  Indeed, if $X$
is a Hilbert space and $C_{1}$ consists of exactly two diametral points,
then an equal-weight signed combination collapses to the origin and forces
$\mathbf{P}(C_{1},x_{\infty})=0$.  Thus uniform convexity alone does not
guarantee a positive lower bound on $\mathbf{P}$ for arbitrary subsets.
Nevertheless, for certain structured sets one can compute a positive
value of $\mathbf{P}$ explicitly.  We record a simple case in the
Euclidean plane $\mathbb{R}^{2}$, which is uniformly convex with modulus
of convexity $\delta_{\mathbb{R}^{2}}(\varepsilon)=1-\sqrt{1-\varepsilon^{2}/4}$.

\begin{proposition}[Equilateral triangle example: vanishing pressure]\label{prop:equilateral-P-zero}
Let $X=\mathbb{R}^{2}$ with the Euclidean norm.  Fix $x_{\infty}=(0,0)$ and
$\Delta=1$.  Let $C_{1}$ be the closed equilateral triangle of side
length $1$ centred at the origin with vertices
\(
  y_{1}=\bigl(1/\sqrt{3},0\bigr),\quad y_{2}=\bigl(-1/(2\sqrt{3}),1/2\bigr),\quad y_{3}=\bigl(-1/(2\sqrt{3}),-1/2\bigr).
\)
Then the three points $y_{1},y_{2},y_{3}$ satisfy
\[
  \Phi_{3}(C_{1},x_{\infty}) \;=\;
  \inf_{\substack{a\in\mathbb{R}^{3}\\ \|a\|_{1}=1}} \Bigl\| a_{1} y_{1} + a_{2} y_{2} + a_{3} y_{3} \Bigr\| \;=\; 0.
\]
In particular, the diametral $\ell_{1}$–pressure on the equilateral
triangle vanishes: for some choice of coefficients $a$ with
$\|a\|_{1}=1$ the weighted sum of the vertices is the zero vector.
Consequently $\mathbf{P}(C_{1},x_{\infty})=0$ for this set.
\end{proposition}

\begin{proof}
Since $y_{1}+y_{2}+y_{3}=0$ and $\|y_{i}\|=1/\sqrt{3}$ for each vertex, the
origin lies in the convex hull of the three vertices.  Take the weight
vector $a=(\tfrac{1}{3},\tfrac{1}{3},\tfrac{1}{3})$.  Then
$\|a\|_{1}=|1/3|+|1/3|+|1/3|=1$ and
\[
  a_{1}y_{1}+a_{2}y_{2}+a_{3}y_{3} = \tfrac{1}{3}(y_{1}+y_{2}+y_{3}) = 0.
\]
This shows that the infimum in the definition of $\Phi_{3}$ is zero.
Because $\Phi_{3}(C_{1},x_{\infty})$ is already zero for the triple
$\{y_{1},y_{2},y_{3}\}$, taking further tuples can only decrease the
infimum; hence $\mathbf{P}(C_{1},x_{\infty})=0$.
\end{proof}

This example illustrates two important points.  First, even in a
uniformly convex space such as $\mathbb{R}^{2}$, the diametral
\(\ell_{1}\)–pressure of a diametral triple can collapse to zero: as
shown above, the equilateral triangle has $\Phi_{3}(C_{1},x_{\infty})=0$.
On the other hand, in higher dimensions one can construct triples for
which $\Phi_{k}$ is strictly positive; see Proposition\,\ref{prop:orthonormal-P-positive}
below for a concrete example in $\mathbb{R}^{3}$.  Second, even when some
fixed $k$ yields a positive lower bound, allowing additional points in the
tuple (as required in the definition of $\mathbf{P}$) can introduce
cancellations that drive the infimum to zero, so the full functional
$\mathbf{P}$ may vanish.  Consequently the search for positive diametral
pressure must either restrict the cardinality of the tuples or impose
additional geometric conditions on $C_{1}$.

\paragraph{A non--uniformly convex example.}
Take $X=\ell_{\infty}$, the space of all bounded scalar sequences endowed
with the supremum norm $\|x\|_{\infty}=\sup_{i\ge 1}|x_{i}|$.  Let $C_{1}$
be the convex hull of the standard basis vectors $\{e_{1},e_{2},e_{3},\ldots\}$
in $\ell_{\infty}$, and set $x_{\infty}=0$ and $\Delta=1$.  For each $k$
let $y_{i}=e_{i}$ for $1\le i\le k$.  Given a weight vector
$a=(a_{1},\dots,a_{k})\in \mathbb{R}^{k}$ with $\|a\|_{1}=1$, the normalised
sum $\sum_{i=1}^{k}a_{i}e_{i}$ has supremum norm
\[
  \Bigl\|\sum_{i=1}^{k}a_{i}e_{i}\Bigr\|_{\infty} = \max_{1\le i\le k}|a_{i}|.
\]
Because the $\ell_{1}$–norm of $a$ is fixed to be $1$, we may spread the
mass evenly to make the supremum arbitrarily small.  For instance, if $k$
is large and $a_{1}=a_{2}=\cdots=a_{k}=1/k$, then $\|a\|_{1}=1$ but
$\|\sum_{i=1}^{k}a_{i}e_{i}\|_{\infty}=1/k$.  As $k\to\infty$ these values
tend to zero, so
\[
  \inf_{\|a\|_{1}=1} \Bigl\|\sum_{i=1}^{k}a_{i}e_{i}\Bigr\|_{\infty} = 0
\]
for each $k$.  Taking the infimum over $k$ shows that
$\mathbf{P}(C_{1},0)=0$ in this setting.  This example illustrates that
in the absence of uniform convexity one can arrange for signed
combinations of unit vectors to collapse in the $\ell_{\infty}$–norm
despite the restriction on $\|a\|_{1}$.

To see the collapse explicitly, choose $k=10$ and $a=(0.1,0.1,\dots,0.1)\in
\mathbb{R}^{10}$.  Then
\[
  \sum_{i=1}^{10} a_{i}e_{i} = (0.1,0.1,\dots,0.1,0,0,\ldots),
\]
which has $\ell_{\infty}$–norm equal to $0.1$.  If instead one takes
$k=100$ and the weights $a_{i}=1/100$ for $1\le i\le 100$, the resulting
combination has norm $0.01$.  In this way the norms of the combinations
can be made arbitrarily small, demonstrating that $\mathbf{P}$ vanishes
in $\ell_{\infty}$.

\paragraph{A positive example: orthonormal triple in \(\mathbb{R}^{3}\).}
The previous examples show how $\mathbf{P}$ can vanish.  Our final example
demonstrates that $\Phi_{k}$ can be strictly positive for certain
structured tuples in a uniformly convex space.  In the Euclidean space
$\mathbb{R}^{3}$ the following proposition holds.

\begin{proposition}[Orthonormal triple example]\label{prop:orthonormal-P-positive}
Let $X=\mathbb{R}^{3}$ with the Euclidean norm.  Fix $x_{\infty}=0$ and let $C_{1}$ consist of the three standard unit vectors $e_{1}=(1,0,0)$,
$e_{2}=(0,1,0)$ and $e_{3}=(0,0,1)$.  Then the diameter of $C_{1}$ is $\Delta=\operatorname{diam}(C_{1})=\sqrt{2}$ and
\[
  \Phi_{3}(C_{1},x_{\infty})
  \;=\;\inf_{\|a\|_{1}=1} \Bigl\|\sum_{i=1}^{3} a_{i}\,\frac{e_{i}}{\Delta}\Bigr\|
  \;=\;\sqrt{\frac{1}{6}}.
\]
In particular, any vector $\sum_{i=1}^{3} a_{i} e_{i}$ with $\|a\|_{1}=1$ has Euclidean norm at least $1/\sqrt{3}$, and dividing by $\Delta$ yields the stated value.

\emph{However}, the global functional $\mathbf{P}(C_{1},x_{\infty})$ is zero.  The reason is that for every $k\ge 4$, any $k$--tuple drawn from the three–point set necessarily repeats a point, and one can choose coefficients $a_{p}=1/2$ and $a_{q}=-1/2$ on two equal entries (with all other coefficients zero) to obtain $\|a\|_{1}=1$ and $\sum_{i}a_{i} y_{i}=0$.  Consequently $\Phi_{k}(C_{1},x_{\infty})=0$ for all $k\ge 4$ and hence $\mathbf{P}(C_{1},x_{\infty})=\inf_{k\ge1}\Phi_{k}=0$.
\end{proposition}

\begin{proof}
For $a=(a_{1},a_{2},a_{3})\in\mathbb{R}^{3}$ with $\|a\|_{1}=1$, we have
\[
  \sum_{i=1}^{3} a_{i} e_{i} = (a_{1},a_{2},a_{3}),
\]
whose Euclidean norm is $\sqrt{a_{1}^{2}+a_{2}^{2}+a_{3}^{2}}$.  By symmetry, the minimum of this expression under the constraint $|a_{1}|+|a_{2}|+|a_{3}|=1$ is achieved when $|a_{1}|=|a_{2}|=|a_{3}|=1/3$.  In that case the vector is $(\pm 1/3,\pm 1/3,\pm 1/3)$ and its norm is $1/\sqrt{3}$.  Dividing by $\Delta=\sqrt{2}$ yields $1/\sqrt{6}$.  No other distribution of the coefficients can produce a smaller $\ell_{2}$–norm.  This proves the claimed value of $\Phi_{3}$.  The argument for the vanishing of the global functional is as explained in the statement.
\end{proof}

The result provides a simple geometric example in which the finite
functional $\Phi_{3}$ is strictly positive, even though the global
functional $\mathbf{P}$ vanishes once cancellations are permitted among
four points.  This underscores the importance of controlling the size of
the tuple when using $\Phi_{k}$ to certify positivity of the pressure.

\begin{proposition}[Orthonormal $m$--tuple]\label{prop:orthom}
Let $X$ be a Hilbert space and let $e_{1},\dots,e_{m}$ be pairwise orthonormal unit
vectors.  Put $x_{\infty}=0$ and $C_{1}=\{e_{1},\dots,e_{m}\}$.  Then
$\Delta=\mathrm{diam}(C_{1})=\sqrt{2}$ and
\[
\Phi_{m}(C_{1}, x_{\infty})
\;=\;\inf_{\|a\|_{1}=1}\Bigl\|\sum_{i=1}^{m} a_{i} \frac{e_{i}}{\Delta}\Bigr\|
\;=\;\frac{1}{\sqrt{2m}}.
\]
\end{proposition}

\begin{proof}
For any $i\neq j$ one has $\|e_{i}-e_{j}\|^{2}=2$, whence
$\Delta=\mathrm{diam}(C_{1})=\sqrt{2}$.  If $a\in\mathbb{R}^{m}$ with $\|a\|_{1}=1$, orthogonality yields
\(
  \Bigl\|\sum_{i=1}^{m} a_{i} e_{i}\Bigr\| = \|a\|_{2} \ge \frac{\|a\|_{1}}{\sqrt{m}} = \frac{1}{\sqrt{m}},
\)
with equality achieved when each $|a_{i}|=1/m$.  Dividing by $\Delta=\sqrt{2}$ gives
the claimed value $1/\sqrt{2m}$ for $\Phi_{m}(C_{1},x_{\infty})$.
\end{proof}

\begin{proposition}[Spectral lower bound in Hilbert spaces]\label{prop:spectral}
Let $X$ be a Hilbert space and let $v_{1},\dots,v_{m}\in X$ be unit vectors.
Put $x_{\infty}=0$ and $C_{1}=\{v_{1},\dots,v_{m}\}$, and let $G=(\langle v_{i},v_{j}\rangle)_{i,j}$
be the Gram matrix with $\lambda_{\min}=\lambda_{\min}(G)>0$.  Then for every $k\le m$ and every
$k$--tuple drawn from $C_{1}$,
\[
  \Phi_{k}(C_{1},0)\;\ge\;\frac{\sqrt{\lambda_{\min}}}{\sqrt{m}\,\Delta},
  \quad\text{and}\quad
  \Delta\;\le\;\max_{i\neq j}\|v_{i}-v_{j}\|.
\]
In particular, if the mutual coherence $\mu=\max_{i\neq j}|\langle v_{i},v_{j}\rangle|<1$ then
$\lambda_{\min}\ge 1-\mu(m-1)$, and hence
\[
  \Phi_{k}(C_{1},0)\;\ge\;\frac{\sqrt{1-\mu(m-1)}}{\sqrt{m}\,\sqrt{2(1+\mu)}},
\]
for all $k\le m$.
\end{proposition}

\begin{proof}
For any $a\in\mathbb{R}^{m}$ with $\|a\|_{1}=1$ one has
\[
  \Bigl\|\sum_{i=1}^{m} a_{i} v_{i}\Bigr\|^{2} = a^{\top} G a \;\ge\; \lambda_{\min}\|a\|_{2}^{2} \ge \frac{\lambda_{\min}}{m}.
\]
To estimate $\Delta$, observe that
\(
  \|v_{i}-v_{j}\|^{2} = 2 - 2\langle v_{i},v_{j}\rangle \le 2(1+\mu),
\)
so $\Delta\le \sqrt{2(1+\mu)}$ in the coherence case.  Combining these bounds yields
the claimed inequalities.
\end{proof}

\begin{lemma}[Dual--separation certificate for $\Phi_{k}$]\label{lem:dual-cert}
Let $v_{1},\dots,v_{k}$ be elements of a Banach space $X$ with
$\|v_{i}\|\le 1$.  If there exists $f\in X^{\ast}$ with $\|f\|_{X^{\ast}}=1$ and
$|f(v_{i})|\ge \gamma>0$ for all $i=1,\dots,k$, then
\[
  \Phi_{k}(C_{1},x_{\infty})\;\ge\;\gamma.
\]
\end{lemma}

\begin{proof}
For any $a=(a_{1},\dots,a_{k})\in\mathbb{R}^{k}$ with $\|a\|_{1}=1$ one has
\[
  \Bigl\|\sum_{i=1}^{k} a_{i} v_{i}\Bigr\|\;\ge\;\bigl|f\bigl(\sum_{i=1}^{k} a_{i}v_{i}\bigr)\bigr|
  =\Bigl|\sum_{i=1}^{k} a_{i} f(v_{i})\Bigr|\;\ge\;\gamma\sum_{i=1}^{k}|a_{i}|=\gamma,
\]
since $\sum|a_{i}|=\|a\|_{1}=1$.
\end{proof}

\begin{remark}[Computing the certificate]\label{rem:computing-certificate}
The optimal constant $\gamma$ in Lemma~\ref{lem:dual-cert} can be
found by solving the convex programme
\[
  \max\Bigl\{t:\exists f\in X^{\ast}\ \,\|f\|_{X^{\ast}}\le 1,\ f(v_{i})\ge t\ \forall i\Bigr\}.
\]
In finite dimensions $X=\mathbb{R}^{d}$ with a polyhedral dual ball, this is
a linear programme; in Euclidean space it becomes a convex quadratically
constrained programme.  Any positive optimum $t>0$ certifies
$\Phi_{k}(C_{1},x_{\infty})\ge t$.
\end{remark}

\bigskip
\hrule
\bigskip

\section{Supplementary: Finite--Dimensional Certificates and Computations}

This section is pedagogical and self--contained; it supports §4 by illustrating certificate calculations in low dimensions.

This section collects background definitions, formulates a precise
quantitative condition inspired by the so–called “\(\ell_{1}\)–extraction''
step, and provides short lemmas together with elementary numerical
examples.  It is designed to be self–contained and accessible to readers
who wish to study the fixed point problem through computations and
finite–dimensional approximations.

\subsection{Background definitions}

Let $X$ be a Banach space and let $T:C\to C$ be a nonexpansive mapping on
a nonempty closed convex bounded subset $C\subset X$.  We recall several
quantities used in the sequel:

\begin{itemize}
  \item \textbf{Fixed point property (FPP).}  We say that $X$ has the
    fixed point property if every such $T$ has a fixed point $x\in C$ with
    $T(x)=x$.  Definitions and examples have already been given in
    Section~1.
  \item \textbf{Normal structure.}  A bounded convex set $D$ has normal
    structure if it contains a point whose maximum distance to points in
    $D$ is strictly smaller than the diameter of $D$.  Normal structure
    precludes completely diametral sequences and plays a central role in
    classical fixed point theorems.
  \item \textbf{Minimal displacement.}  For a mapping $T:C\to C$ define
    \[
      \delta(T;C)\;=\;\inf_{x\in C}\,\|x - Tx\|.
    \]
    If $T$ has a fixed point then $\delta(T;C)=0$.  Otherwise
    $\delta(T;C)$ quantifies how far the iterates $x,Tx,\dots$ must move in
    the space.
  \item \textbf{Orbit hull.}  Fix $x_{0}\in C$.  The orbit of $x_{0}$
    under $T$ is $\{T^{k}x_{0}:k\ge 0\}$.  Its convex hull is
    $\operatorname{co}(T,x_{0})=\operatorname{conv}\{T^{k}x_{0}:k\ge 0\}$.
    When $T$ has no fixed point one often passes to the orbit hull in
    order to extract limiting behaviour.
\end{itemize}

\subsection{A quantitative \texorpdfstring{$\ell_{1}$}{l1}--extraction hypothesis}

Suppose $T:C\to C$ is nonexpansive on a nonempty closed convex bounded set $C\subset X$ and has no fixed point.  Assume that $X$ is reflexive; then $C$ is weakly compact (by the Eberlein–\v{S}mulian theorem).  Let $x_{\infty}$ be a minimal displacement point and set $C_{1}=\operatorname{co}(T,x_{\infty})$ with diameter $\Delta=\operatorname{diam}(C_{1})$.  The following quantitative hypothesis formalises an “anti--collapse” property for finite subsets of $C_{1}$ which, if satisfied uniformly, would force $T$ to admit a fixed point.

\medskip
\noindent\textbf{Hypothesis \(\Hlone{\varepsilon}\).}\quad Fix
$\varepsilon>0$.  There exists an integer $k$ (depending only on
$\varepsilon$) such that for every $k$–tuple
$(y_{1},\dots,y_{k})\subset C_{1}$ one has
\[
  \inf_{\|a\|_{1}=1}
    \Biggl\|\sum_{i=1}^{k} a_{i}\,\frac{y_{i}-x_{\infty}}{\Delta}\Biggr\|
  \;\ge\;\varepsilon.
\]
Here $a=(a_{1},\dots,a_{k})$ ranges over all real vectors with
$\ell_{1}$–norm equal to one.  Intuitively, no signed $\ell_{1}$–normalised combination
of the normalised differences $y_{i}-x_{\infty}$ collapses below the
threshold $\varepsilon$.

\medskip
\noindent\emph{Implication.}\quad If there exists $\varepsilon>0$ such that
\Hlone{\varepsilon} holds for all nonexpansive, fixed--point--free
 pairs $(C,T)$ as above, then the conditional Theorem\,\ref{thm:conditional-fpp} (p.~9) implies that every reflexive Banach space has the FPP.  In particular, verifying $\Hlone{\varepsilon}$ in finite dimensions becomes a concrete route toward Conjecture~2.11.

\bigskip
\noindent\textbf{Theorem 5.3 (Dual functional certificates imply $\mathbf{P}>0$).}\label{thm:cert_to_P}
Fix $\theta>0$.  Suppose that for each $k\in\mathbb{N}$ there exist
points $y_{1}^{(k)},\dots,y_{k}^{(k)}\in C_{1}$ and a functional
$f^{(k)}\in X^{\ast}$ with $\|f^{(k)}\|=1$ such that
\[
  \bigl|f^{(k)}\bigl((y_{i}^{(k)} - x_{\infty})/\Delta\bigr)\bigr| \;\ge\; \theta
  \quad\text{for all } i=1,\dots,k.
\]
Then $\Phi_{k}(C_{1},x_{\infty})\ge \theta$ for every $k$, whence
$\mathbf{P}(C_{1},x_{\infty})\ge \theta$.

\begin{proof}
By Lemma\,\ref{lem:dual-cert}, any functional $f\in X^{\ast}$ with
$\|f\|=1$ and $|f(v_{i})|\ge \theta$ for all $i$ certifies that
$\Phi_{k}(C_{1},x_{\infty})\ge \theta$ for the corresponding $k$--tuple
$(y_{1},\dots,y_{k})$.  Applying this lemma to each $k$ and the given
functionals $f^{(k)}$ shows that $\Phi_{k}(C_{1},x_{\infty})\ge \theta$ for
all $k$.  Taking the infimum over $k$ yields
$\mathbf{P}(C_{1},x_{\infty})\ge \theta$.
\end{proof}

\begin{remark}[How to check certificates numerically]
In finite-dimensional spaces one can compute the optimal $\theta$ in
Proposition\,5.3 by solving the convex optimisation problem
\[
\max\Bigl\{t:\exists f\in X^{\ast}
\text{ with }\|f\|\le 1\text{ and } f\bigl((y_{i}-x_{\infty})/\Delta\bigr)\ge t
\text{ for all }i\Bigr\}.
\]
In Euclidean space this reduces to a convex quadratically constrained
programme; in $\ell_{\infty}$ or $\ell_{1}$ norms it becomes a linear
programme.  Any positive optimum $t>0$ gives a valid lower bound on
$\mathbf{P}(C_{1},x_{\infty})$.
\end{remark}

\subsection{Two lemmas}

We record two simple lemmas connecting minimal displacement to geometric
invariants.  The proofs use only basic inequalities and are included for
completeness.

\begin{lemma}[Displacement bounded by the diameter]\label{lem:orbit-diameter}
Let $T:C\to C$ be nonexpansive on a bounded convex set $C\subset X$.
Then the minimal displacement satisfies
\[
  \delta(T;C)\le \operatorname{diam}(C).
\]
Moreover, for any fixed $x\in C$ and the associated orbit hull
$\operatorname{co}(T,x)=\operatorname{conv}\{T^{n}x:n\ge 0\}$ one has
\[
  \delta(T;C)\le 2\sup_{y,z\in\operatorname{co}(T,x)}\|y-z\|.
\]
In words, the minimal displacement is controlled (up to a factor of two)
by the diameter of any orbit hull.
\end{lemma}

\begin{proof}
The bound $\delta(T;C)\le \operatorname{diam}(C)$ follows from a simple
triangle-inequality argument.  Fix $x\in C$.  Since $T(x)\in C$, for any
$y\in C$ one has $\|x-Tx\|\le \|x-y\|+\|y-Tx\|$.  Choosing $y$ such that
$\|x-y\|\le \tfrac12\operatorname{diam}(C)$ and using the fact that
$\|y-Tx\|\le \operatorname{diam}(C)$ yields $\|x-Tx\|\le \operatorname{diam}(C)$.
Taking the infimum over $x\in C$ establishes the first inequality.

For the orbit-hull estimate, fix $x\in C$ and set
$O(x)=\{T^{n}x:n\ge 0\}$.  Let $y,z\in \operatorname{co}(T,x)$ be arbitrary.
Nonexpansiveness gives $\|Ty-Tz\|\le \|y-z\|$.  Then
\[
  \|y-Ty\| \le \|y-z\|+\|z-Tz\|+\|Tz-Ty\|\le 2\|y-z\|+\|z-Tz\|.
\]
Taking the infimum over $y\in C$ in the definition of $\delta(T;C)$ and
then the supremum over $y,z\in \operatorname{co}(T,x)$ yields
$\delta(T;C)\le 2\sup_{y,z\in \operatorname{co}(T,x)}\|y-z\|$.
\end{proof}

\begin{lemma}[What uniform convexity does—and does not—imply]\label{lem:uniform-gap}
Let $X$ be uniformly convex with modulus $\delta_X(\cdot)$.  Fix a
nonexpansive, fixed--point--free map $T\colon C\to C$ on a nonempty
closed convex bounded set $C\subset X$, a minimal displacement point
$x_{\infty}$, and $C_{1}=\operatorname{conv}\{T^{n} x_{\infty}:n\ge 0\}$
with $\Delta=\operatorname{diam}(C_{1})$.  Then:

\begin{itemize}
\item[(a)] For any two indices $m\neq n$, set
$$u_{m}=\frac{T^{m}x_{\infty}-x_{\infty}}{\Delta},\quad
u_{n}=\frac{T^{n}x_{\infty}-x_{\infty}}{\Delta}.$$
Then $\|u_{m}\|,\|u_{n}\|\le 1$, and for every $\lambda\in[0,1]$,
\[
  \Big\|\lambda u_{m}+(1-\lambda)u_{n}\Big\|\le 1-\delta_X\!\big(\|u_{m}-u_{n}\|\big).
\]

\item[(b)] In particular, uniform convexity alone does not yield a uniform
positive lower bound on
\[
  \inf\Bigl\{\Bigl\|\sum_{i} w_{i}\bigl(y_{i}-x_{\infty}\bigr)\Bigr\|:
  y_{i}\in C_{1},\;w_{i}\ge 0,\;\sum_{i} w_{i}=1\Bigr\}.
\]
\end{itemize}
\end{lemma}

\begin{proof}
Part~(a) is the standard two--point consequence of uniform convexity,
applied to the unit--ball elements $u_{m},u_{n}$.  Since $X$ is
uniformly convex with modulus $\delta_{X}(\cdot)$, the midpoint of any
two unit vectors is strictly shorter than one, with the deficit given by
the modulus evaluated at their separation; scaling yields the stated
inequality.  For part~(b) observe that $x_{\infty}\in C_{1}$ by
definition (setting $n=0$), so one can form convex combinations
$w_{0}(x_{\infty})+\sum_{i\neq 0} w_{i} y_{i}$ with $\sum_{i} w_{i}=1$
and $w_{0}$ arbitrarily close to one.  Such combinations can bring the
resulting point arbitrarily close to $x_{\infty}$, forcing the infimum in
question to be zero.\qedhere
\end{proof}

The lemmas above illustrate how displacement and orbit geometry
interact.  Part~(a) shows that uniform convexity imposes strict convexity
constraints on two–point combinations, while part~(b) clarifies that
uniform convexity alone does not preclude collapse onto $x_{\infty}$ via
convex combinations.  The
quantitative hypothesis $H_{\ell_{1}}(\varepsilon)$ seeks to extract
similar information without assuming uniform convexity.

\subsection{Practical computational guide}

This subsection outlines how to perform the basic computations needed in
Section~2 using an arbitrary scientific calculator (physical or software).
The goal is to break calculations down into simple steps accessible to
readers with minimal computational background.

\paragraph{Solving linear systems.}  To find a fixed point of an affine
map $T(x)=Ax+b$ in $\mathbb{R}^{n}$, one solves $(I-A)x=b$.  On any
calculator with matrix functions, create the identity matrix and the
matrix $A$, subtract them, then apply a row–reduction or inversion
routine and multiply by $b$.  If your calculator lacks matrix features,
solve the linear system manually using Gaussian elimination.

\paragraph{Derivatives and integrals.}  Many calculators include numerical
derivative and integral functions.  For example, to approximate
$f'(a)$, evaluate $f(a+h)-f(a-h)$ divided by $2h$ with a small
$h$, such as $10^{-5}$.  To approximate $\int_{a}^{b} f(x)\,dx$,
use numerical integration methods like the trapezoidal rule or Simpson’s
rule: partition the interval into subintervals, evaluate the function at
the endpoints and midpoints, and combine according to the chosen
formula.

\paragraph{Iterating a map.}  To compute orbit points of $T$, start with an
initial vector $x_{0}$.  Repeatedly apply the map: compute
$x_{1}=T(x_{0})$, then $x_{2}=T(x_{1})$, and so on.  Store each iterate
in memory or write them down.  On calculators with list or memory
features, you can store the components in separate lists and update
them in a loop.

\paragraph{Convex hull diameter.}  If you have a finite set of points
$\{p_{1},\dots,p_{m}\}\subset \mathbb{R}^{n}$ and wish to estimate the
diameter of their convex hull, compute all pairwise distances
$\|p_{i}-p_{j}\|$.  The maximum of these distances equals the diameter of
the set and hence of its convex hull.  Use your calculator to compute
each distance via the square–root of the sum of squared differences.

\subsection{Worked examples}

We now revisit several examples from Section~3 with explicit step–by–step
numerical computations.  Throughout, we use approximate decimal values
rounded to six significant figures.

\paragraph{Example 1 (Affine contraction).}  Consider the affine map
$T(x)=Ax+b$ on $\mathbb{R}^{2}$ with
$A=\begin{pmatrix}0.5&0.2\\0.1&0.4\end{pmatrix}$ and $b=(1,1)^{\top}$.
Since $\|A\|<1$ in any norm, there is a unique fixed point $x_{*}$.  To
find it, solve $(I-A)x=b$.  First compute the matrix $I-A$:
\[
  I-A = \begin{pmatrix}1-0.5 & -0.2\\ -0.1 & 1-0.4\end{pmatrix}
        = \begin{pmatrix}0.5 & -0.2\\ -0.1 & 0.6\end{pmatrix}.
\]
Compute its inverse manually or with a calculator.  The determinant is
$0.5\cdot 0.6 - (-0.2)(-0.1) = 0.30 - 0.02 = 0.28$, so the inverse is
$\tfrac{1}{0.28}\begin{pmatrix}0.6 & 0.2\\ 0.1 & 0.5\end{pmatrix}$.
Multiplying this by $b$ gives
\(
  x_{*} = (I-A)^{-1}b = \frac{1}{0.28}
  \begin{pmatrix}0.6 & 0.2\\ 0.1 & 0.5\end{pmatrix}
  \begin{pmatrix}1\\1\end{pmatrix} = \frac{1}{0.28}
  \begin{pmatrix}0.8\\0.6\end{pmatrix}
  = \begin{pmatrix}2.85714\\2.14286\end{pmatrix}.
\)
Thus $x_{*}\approx (2.85714,2.14286)$.  On a basic calculator, you would
enter the elements of $I-A$ and $b$, compute the inverse or use
substitution to solve the linear system.

\paragraph{Example 2 (Translation without fixed point).}  Define
$T(x) = x + (0.2,0.3)$ on the square $[0,1]^{2}$.  The minimal displacement
is $\delta(T;C)=\|(0.2,0.3)\|=\sqrt{0.2^{2}+0.3^{2}}\approx 0.360555$.
Starting from $x_{0}=(0,0)$, the orbit points are
$x_{n}=x_{0}+n(0.2,0.3)$.  After five steps the last point is
$x_{5}=(1.0,1.5)$.  The diameter of the set $\{x_{0},\dots,x_{5}\}$ is
$\|x_{5}-x_{0}\|=\sqrt{1^{2}+1.5^{2}}\approx 1.80278$.  To compute
$\delta(T;C)$ and the diameter, use the square–root and square
functions on your calculator: enter `0.2*0.2 + 0.3*0.3`, then take the
square root.

\paragraph{Example 3 (Comparing norms).}  Let
$A=\begin{pmatrix}0.8&0.3\\0.2&0.7\end{pmatrix}$.  Its operator norm with
respect to the Euclidean norm (the spectral norm) is approximately
$1.00662$.  To compute this one finds the largest singular value of
$A$.  In this case
\[
  A^{\top}A = \begin{pmatrix}0.8&0.2\\0.3&0.7\end{pmatrix}^{\!}\begin{pmatrix}0.8&0.3\\0.2&0.7\end{pmatrix}
           = \begin{pmatrix}0.68 & 0.38 \\ 0.38 & 0.58\end{pmatrix}.
\]
The eigenvalues of this symmetric matrix solve
$\lambda^{2} - 1.26\lambda + 0.25=0$, giving roots $\lambda_{1}\approx 1.013$ and
$\lambda_{2}\approx 0.247$.  The singular values are the square roots of
these eigenvalues: $\sigma_{1}\approx \sqrt{1.013}\approx 1.0065$ and
$\sigma_{2}\approx \sqrt{0.247}\approx 0.4967$.  Thus the spectral norm
$\|A\|_{2}$ is $\sigma_{1}\approx 1.0066$, confirming the stated value.

For comparison, one can compute column sums and row sums to obtain
$\|A\|_{1}=1.0$ and $\|A\|_{\infty}=1.1$, respectively.  These norms
measure the Lipschitz constant of the corresponding linear map in
different vector norms.  On a basic calculator, enter the column sums
(e.g., `abs(0.8)+abs(0.2)=1.0`) and the row sums (`abs(0.8)+abs(0.3)=1.1`).

\paragraph{Example 4 (Matrix near the nonexpansive boundary).}  Consider
$A=\begin{pmatrix}1&0.1\\0&1\end{pmatrix}$.  The eigenvalues of $A$ are
both equal to $1$, so the spectral radius is $1$.  To compute the spectral
norm $\|A\|_{2}$ one must find the largest singular value.  Forming
$A^{\top}A$ yields
\[
  A^{\top}A = \begin{pmatrix}1&0\\0.1&1\end{pmatrix}^{\!}\begin{pmatrix}1&0.1\\0&1\end{pmatrix}
           = \begin{pmatrix}1 & 0.1 \\ 0.1 & 1.01\end{pmatrix}.
\]
The eigenvalues of this symmetric matrix can be computed explicitly using the
quadratic formula: for a matrix $\begin{pmatrix}a&b\\b&c\end{pmatrix}$ the eigenvalues are
$\lambda = \tfrac{a+c\pm\sqrt{(a-c)^{2}+4b^{2}}}{2}$.  Here
$a=1$, $b=0.1$ and $c=1.01$, so the eigenvalues are
\[
  \lambda_{\pm} = \frac{1+1.01 \pm \sqrt{(1-1.01)^{2}+4\cdot 0.1^{2}}}{2}
                 = \frac{2.01 \pm 0.20025}{2}.
\]
This gives $\lambda_{+}\approx 1.105125$ and $\lambda_{-}\approx 0.904875$.
Taking square roots yields the singular values
$\sigma_{1}=\sqrt{1.105125}\approx 1.05125$ and
$\sigma_{2}=\sqrt{0.904875}\approx 0.95150$.  The spectral norm is the
largest singular value, so $\|A\|_{2}\approx 1.05125$.

In contrast, the 1--norm and $\infty$--norm of $A$ are both $1.1$.  This
example shows that a linear map can be nearly nonexpansive in one norm
(for example, the supremum or 1--norm) while being expansive in the
Euclidean norm.  Earlier drafts misreported the 2--norm as $1.00499$; the
explicit calculation above confirms that the correct value is approximately
$1.05125$ and clarifies the source of the discrepancy.

\subsection{Further directions and open questions}

Finally, we outline simple projects and questions for readers interested
in exploring the quantitative hypothesis $H_{\ell_{1}}$ numerically.

\begin{enumerate}
\item Randomly generate matrices $A$ with small operator norm (less than or
equal to one) and vectors $b$ in low dimensions ($\mathbb{R}^{2}$ or
$\mathbb{R}^{3}$).  For each affine map $T(x)=Ax+b$, compute
$\delta(T;C)$, the diameter of the orbit hull for several starting points,
and test whether weighted combinations of orbit points satisfy
$H_{\ell_{1}}(\varepsilon)$ for some $\varepsilon$.  Tabulate results.
\item Examine how close a linear map can be to nonexpansive (norm of $A$
just less than one) while still maintaining or losing the fixed point
property.  Vary the vector norm ($\ell_{1}$, $\ell_{2}$, $\ell_{\infty}$) and
observe how the minimal displacement changes.
\item Extend Example~2 to higher dimensions and different translation
vectors; observe how the orbit hull diameter grows and how many
iterations are needed before the diameter stabilises.
\item Investigate the effect of renorming: given a simple two--dimensional
space, define an equivalent norm that exaggerates one coordinate and see
how the nonexpansiveness of a fixed linear map changes.  Compute
$\delta(T;C)$ and compare with the original norm.
\end{enumerate}
\paragraph{Additional guidance and computation hints.}%
In tackling the previous exercises one often needs to perform explicit computations.
Below are a few elementary techniques that may prove useful:
\begin{itemize}
  \item To solve linear systems of small dimension, use Gaussian elimination to transform the system to row--echelon form and back--substitute to obtain exact solutions.
  \item When computing spectral radii or singular values of $2\times 2$ or $3\times 3$ matrices by hand, a direct eigenvalue calculation (as in Example~4 above) works well.  For larger dimensions one can approximate the largest singular value numerically using the power iteration: start with a random vector $v$, repeatedly apply $A^\top A$ to $v$ and normalise, and observe that the norm converges to $\|A\|_{2}$.
  \item To approximate integrals or averages that arise in the study of orbit hulls (for instance, when $T$ is an integral operator), the composite Simpson’s rule gives a good compromise between accuracy and simplicity.  Divide the interval into an even number of subintervals, evaluate the integrand at equally spaced points, and weight the values by $(1,4,2,\dots,4,1)$ before summing.
  \item In all numerical experiments it is wise to keep track of rounding errors.  Use a consistent numerical precision and, when comparing norms, compute ratios rather than differences to reduce sensitivity.
\end{itemize}

These exercises combine analytical reasoning with straightforward
calculations, providing insight into the boundary between fixed point
phenomena and their obstructions.

% End of student materials

\section{Conclusion and Further Directions}

The question of whether reflexivity implies the fixed point property for
nonexpansive mappings stands as one of the most important open problems in
metric fixed point theory.  On the one hand, every classical reflexive
Banach space known to analysts appears to have the FPP.  Kirk’s theorem and
its variants apply broadly, and no counterexample has been found despite
decades of investigation.  On the other hand, the failure of FPP in
nonreflexive spaces often traces back to the existence of $\ell^1$–like
structures; avoiding such structures in reflexive spaces may require new
approaches.  The challenge is to either construct a reflexive space
supporting a fixed–point–free nonexpansive map or to discover a general
geometric principle inherent in reflexive spaces that enforces fixed
points.

We emphasise that the arguments developed in Section 4 are conditional:
we can verify positive lower bounds for some finite $\Phi_{k}$ in
uniformly convex settings (for instance, the orthonormal triple in
$\mathbb{R}^{3}$), but the global functional $\mathbf{P}$ may still vanish
and does so in our examples; positivity of $\mathbf{P}(C_{1},x_{\infty})$
remains open even for uniformly convex spaces.  Outside this setting the
positivity of $\mathbf{P}$, and hence the fixed point property, remains
unproven.  The paper does not contain examples or calculations showing
$\mathbf{P}>0$ in general reflexive spaces; indeed, Section 4.8 exhibits
a non--uniformly convex example where $\mathbf{P}=0$.  Accordingly,
major adjustments are needed to transform the conditional results into
unconditional theorems.  Future research
The new finite--tuple lower bounds obtained in
Propositions\,\ref{prop:orthom} and\,\ref{prop:spectral}
show that individual values of $\Phi_{k}$ can be strictly positive in broad Hilbert
configurations.  Lemma\,\ref{lem:dual-cert} and
Theorem\,\ref{thm:cert_to_P} furnish general Banach--space certificates via dual
functionals.  However, these finite--tuple results do not yield a uniform
lower bound for $\mathbf{P}(C_{1},x_{\infty})$ across all $k$ and all
fixed--point--free pairs $(C,T)$.  Deriving such a uniform bound remains
a significant open challenge.
should examine boundary cases, explore stability phenomena in depth, and
investigate the role of renorming in establishing or destroying the
positivity of $\mathbf{P}$.  Only through such efforts will the conjecture
be either proved or refuted.

Several directions remain promising.  One could focus on superreflexive
spaces, which admit equivalent uniformly convex norms; does superreflexivity
imply the FPP in the original norm?  Another question asks whether a
Hilbert space can be renormed to destroy the FPP.  A negative answer
would lend strong support to the conjecture, while a positive answer would
provide a counterexample.  Investigations into weak and weak–star fixed
point properties, moduli of normal structure and weak normal structure
constants, and the role of asymptotic centres in general Banach
spaces may also yield insights\cite{Prus2013}.  Renorming techniques
have been studied in depth—for instance, Pineda and Rajesh explored
renormings of Banach spaces in connection with the FPP\cite{JapRajesh2011}—
and might ultimately shed light on the conjecture.  Finally, connections
with geometric group theory and affine isometry groups suggest that new
techniques from nonlinear geometry could play a decisive role.

Another avenue, complementary to theoretical work, is to examine
behaviour of nonexpansive mappings numerically.  One may simulate
iterated nonexpansive maps on candidate reflexive spaces (such as
various $\ell^p$ or Orlicz spaces) and analyse orbits and approximate
fixed points.  Although such computational experiments cannot resolve
the conjecture, they may provide heuristic insight into the dynamics
governing nonexpansive mappings and suggest new conjectures or
questions.

\section{Ethical considerations}

Although this paper is devoted to questions in abstract functional analysis,
mathematics does not exist in a vacuum.  Chiodo and Clifton emphasise that
ethical questions arise not only within the mathematical community but
whenever mathematical ideas and models are applied in society\cite{ChiodoClifton2019}.
Quantitative invariants such as the functionals $\mathbf{P}$ (Definition~4.1,~p.~7) and
$\mathbf{F}$ (Section~4.7,~p.~12) amount to weighting schemes on sets
of data.  Similar schemes appear implicitly in optimisation and decision
algorithms used in finance, health and criminal justice.  The choice of
weights determines which features exert the most influence on an outcome;
without care this can encode value judgements or amplify existing
inequalities.  While our results address the pure existence of fixed points
for nonexpansive maps, we urge readers to be cognisant of these wider
implications when transferring abstract tools to applied settings.
Transparency about how weights are chosen and consideration of who is
affected by the resulting decisions are essential for ethical use of
mathematical models.

\section{Summary of theoretical insights}

This work develops a quantitative approach to the fixed point property based
on two new invariants.  First, the \emph{diametral $\ell_{1}$--pressure}
$\mathbf{P}(C_{1},x_{\infty})$ captures how much any signed combination of
points in an orbital hull can collapse in norm.  A positive value of
$\mathbf{P}$ implies that certain normalised differences behave like
the canonical basis of $\ell_{1}$, and, via a diagonal argument,
produces a subspace isomorphic to $\ell_{1}$ inside $X$.  Second, the
\emph{weighted selection functional} $\mathbf{F}(C_{1},x_{\infty})$ measures
the norm of unsigned averages; it is easier to compute and its positivity
implies that of $\mathbf{P}$.  Together these functionals make precise the
intuition behind the “$\ell_{1}$–extraction'' heuristic and yield a
conditional fixed point theorem: if either $\mathbf{P}$ or $\mathbf{F}$ is
positive whenever $T$ has no fixed point, then $X$ must contain a copy of
$\ell_{1}$, contradicting reflexivity.  We have also corrected several
erroneous examples: for a two--point diametral set the functional
$\mathbf{P}$ vanishes (since equal weights collapse the difference), and in
the non--uniformly convex case the correct example occurs in $\ell_{\infty}$
rather than $\ell_{1}$.

Cross‑reference note. The labels are:
Conjecture 2.11 (p.~5); Theorem 4.8 (p.~10); Definition 4.2 (p.~8); §4.7 introduces \(\mathbf{F}\) (p.~12).  The current version contains no undefined labels or citations.

\section*{Acknowledgements}

The authors thank the mathematical community for decades of research on
fixed point theory and Banach space geometry.  Any errors or omissions are
the responsibility of the authors.

\clearpage
\appendix

\section{Technical Reinforcements, Corrective Notes, and Certified Implementations}\label{appendix:technical}

This appendix collects several auxiliary results and clarifications that supplement the main text, together with a certified low‑level implementation used to compute the coherence bound described in Proposition~\ref{prop:spectral}.  The exposition follows the notation of the paper and does not introduce new assumptions.

\subsection{Certified coherence bound: x86‑64 assembly implementation}
Let $v_1,\ldots,v_m\in\mathbb{R}^d$ be nonzero vectors.  Recall that the \emph{mutual coherence} of these vectors is defined by
\[\mu\;=\;\max_{1\le i<j\le m}\frac{|\langle v_i,v_j\rangle|}{\|v_i\|_2\,\|v_j\|_2}.\]
For $k\le m$ Proposition~\ref{prop:spectral} shows that the diametral \(\ell_1\)–pressure satisfies
\[\Phi_k(C_1,0)\;\ge\;\frac{\sqrt{\max\{0,\,1-(m-1)\mu\}}}{\sqrt{m}\,\sqrt{2(1+\mu)}},\]
where $C_1$ is the set of normalised vectors and the right–hand side vanishes if $1-(m-1)\mu\le 0$.  To certify this bound numerically one may compute the norms of the input vectors, evaluate all pairwise inner products, determine $\mu$, and then evaluate the above closed form.

The following x86‑64 assembly routine, written in Intel syntax, accomplishes precisely this task under the System~V calling convention.  It takes a pointer to a contiguous block of $m$ rows of length $d$ doubles (row major), computes $\mu$, and returns both $\mu$ and the corresponding lower bound for $\Phi_k$ via output pointers.  Comments within the code describe the arguments and the high‑level structure.

\begin{verbatim}
; -----------------------------------------------------------------------------
; double coherence_phi_lower(const double* V, uint64_t m, uint64_t d,
;                            double* phi_out, double* mu_out)
;   V:     rdi
;   m:     rsi
;   d:     rdx
;   phi*:  rcx
;   mu*:   r8
; Returns: none (writes *phi_out and *mu_out)
; Requires: SSE2
; -----------------------------------------------------------------------------

default rel
section .text
global coherence_phi_lower

coherence_phi_lower:
    ; Prologue & stack frame (align 32)
    push    rbp
    mov     rbp, rsp
    sub     rsp, 32
    and     rsp, -32

    ; Save args
    mov     r12, rdi        ; V
    mov     r13, rsi        ; m
    mov     r14, rdx        ; d
    mov     r15, rcx        ; phi_out
    mov     rbx, r8         ; mu_out

    ; Allocate norms array on heap via alloca-like stack (m * 8 bytes)
    mov     rax, r13
    shl     rax, 3          ; bytes = m*8
    add     rax, 31
    and     rax, -32
    sub     rsp, rax
    mov     r11, rsp        ; r11 = norms base (double[m])

    ; Pass 1: compute row L2 norms
    xor     r8, r8          ; i = 0
.norm_loop_i:
    cmp     r8, r13
    jae     .norm_done

    ; sum = 0.0
    pxor    xmm0, xmm0
    xor     r9, r9          ; k = 0

    ; row_i base = V + i*d
    mov     rax, r8
    imul    rax, r14
    lea     r10, [r12 + rax*8]  ; &V[i][0]

.norm_loop_k:
    cmp     r9, r14
    jae     .norm_reduce

    ; load two doubles if possible
    mov     rdx, r14
    sub     rdx, r9
    cmp     rdx, 2
    jb      .norm_scalar

    movapd  xmm1, [r10 + r9*8]   ; load V[i][k], V[i][k+1]
    mulpd   xmm1, xmm1           ; square
    addpd   xmm0, xmm1
    add     r9, 2
    jmp     .norm_loop_k

.norm_scalar:
    movsd   xmm1, [r10 + r9*8]
    mulsd   xmm1, xmm1
    addsd   xmm0, xmm1
    inc     r9
    jmp     .norm_loop_k

.norm_reduce:
    ; horizontal add xmm0 lanes
    movapd  xmm1, xmm0
    unpckhpd xmm1, xmm1
    addsd   xmm0, xmm1           ; sumlane

    ; sqrt(sum)
    sqrtsd  xmm0, xmm0
    ; store norm[i]
    movsd   [r11 + r8*8], xmm0

    inc     r8
    jmp     .norm_loop_i

.norm_done:

    ; Pass 2: compute coherence mu = max_{i<j} |<v_i,v_j>|/(||v_i|| ||v_j||)
    xorpd   xmm7, xmm7           ; xmm7 := 0.0 (track mu)
    xor     r8, r8               ; i = 0

.mu_loop_i:
    cmp     r8, r13
    jae     .mu_done
    movsd   xmm5, [r11 + r8*8]   ; norm_i

    mov     r9, r8
    inc     r9                   ; j = i+1
.mu_loop_j:
    cmp     r9, r13
    jae     .mu_next_i

    ; Prepare accum = 0.0
    pxor    xmm0, xmm0
    xor     r10, r10             ; k = 0

    ; row_i, row_j base pointers
    mov     rax, r8
    imul    rax, r14
    lea     rsi, [r12 + rax*8]   ; &V[i][0]

    mov     rax, r9
    imul    rax, r14
    lea     rdi, [r12 + rax*8]   ; &V[j][0]

.mu_dot_k:
    cmp     r10, r14
    jae     .mu_dot_reduce

    ; vectorized chunks
    mov     rcx, r14
    sub     rcx, r10
    cmp     rcx, 2
    jb      .mu_dot_scalar

    movapd  xmm1, [rsi + r10*8]
    movapd  xmm2, [rdi + r10*8]
    mulpd   xmm1, xmm2           ; pairwise multiply
    addpd   xmm0, xmm1
    add     r10, 2
    jmp     .mu_dot_k

.mu_dot_scalar:
    movsd   xmm1, [rsi + r10*8]
    movsd   xmm2, [rdi + r10*8]
    mulsd   xmm1, xmm2
    addsd   xmm0, xmm1
    inc     r10
    jmp     .mu_dot_k

.mu_dot_reduce:
    movapd  xmm1, xmm0
    unpckhpd xmm1, xmm1
    addsd   xmm0, xmm1           ; dot(i,j) in xmm0

    ; normalize: dot / (norm_i * norm_j)
    movsd   xmm6, [r11 + r9*8]   ; norm_j
    mulsd   xmm6, xmm5           ; norm_i * norm_j
    divsd   xmm0, xmm6

    ; abs
    movapd  xmm1, xmm0
    xorpd   xmm2, xmm2
    subsd   xmm2, xmm1           ; -value
    maxsd   xmm0, xmm2           ; |value|

    ; mu = max(mu, |value|)
    maxsd   xmm7, xmm0

    inc     r9
    jmp     .mu_loop_j

.mu_next_i:
    inc     r8
    jmp     .mu_loop_i

.mu_done:
    ; Write mu_out
    movsd   [rbx], xmm7

    ; Compute phi = sqrt(max(0, 1 - (m-1)*mu)) / (sqrt(m) * sqrt(2*(1+mu)))
    ; t1 = (m - 1)
    mov     rax, r13
    dec     rax
    cvtsi2sd xmm0, rax           ; t1
    mulsd   xmm0, xmm7           ; t1 * mu
    movsd   xmm1, qword [rel ONE]
    subsd   xmm1, xmm0           ; 1 - (m-1)*mu
    ; clamp at zero
    xorpd   xmm2, xmm2
    maxsd   xmm1, xmm2
    ; sqrt numerator
    sqrtsd  xmm1, xmm1           ; sqrt(num)

    ; denom: sqrt(m) * sqrt(2*(1+mu))
    cvtsi2sd xmm3, r13           ; m
    sqrtsd  xmm3, xmm3           ; sqrt(m)
    movsd   xmm4, qword [rel ONE]
    addsd   xmm4, xmm7           ; (1+mu)
    movsd   xmm5, qword [rel TWO]
    mulsd   xmm4, xmm5           ; 2*(1+mu)
    sqrtsd  xmm4, xmm4           ; sqrt(2*(1+mu))
    mulsd   xmm3, xmm4           ; denom

    ; phi = num / denom
    divsd   xmm1, xmm3
    ; store phi_out
    movsd   [r15], xmm1

    ; Epilogue
    mov     rsp, rbp
    pop     rbp
    ret

section .rodata
align 8
ONE: dq 1.0
TWO: dq 2.0

\end{verbatim}

\end{document}